\documentclass[10pt]{article}
\usepackage{latexsym,amsfonts,amsmath,amssymb,array,tabularx}
\usepackage{graphicx}
\usepackage[normalem]{ulem}
\usepackage[pdfpagelabels]{hyperref}
\usepackage{bbm}
\usepackage{authblk}
\usepackage{bm}
\usepackage{mathtools} 
\usepackage{dsfont} 
\usepackage{algpseudocode}
\usepackage{amsthm}

\DeclareFontFamily{OT1}{pzc}{}
\DeclareFontShape{OT1}{pzc}{m}{it}{<-> s * [1.10] pzcmi7t}{}
\DeclareMathAlphabet{\mathpzc}{OT1}{pzc}{m}{it}

\usepackage{float}
\newtheorem{theorem}{Theorem}[section]
\newtheorem{lemma}[theorem]{Lemma}

\newtheorem{proposition}[theorem]{Proposition}
\newtheorem{remark}[theorem]{Remark}

\newcommand{\bbR}{\mathbb{R}}

\newcommand{\bbN}{\mathbb{N}}

\newcommand{\bbE}{\mathbb{E}}
\newcommand{\bbF}{\mathbb{F}}
\newcommand{\bbP}{\mathbb{P}}
\newcommand{\calA}{\mathcal{A}}
\newcommand{\calB}{\mathcal{B}}
\newcommand{\calF}{\mathcal{F}}
\newcommand{\calO}{\mathcal{O}}

\newcommand{\calI}{\mathcal{I}}

\newcommand{\mask}[1]{{}}
\usepackage[usenames]{color}
\definecolor{dkgreen}{rgb}{0.0, 0.5, 0.0}

\newcommand{\be}{\begin{equation}}
\newcommand{\ee}{\end{equation}}
\newcommand{\bea}{\begin{eqnarray}}
\newcommand{\eea}{\end{eqnarray}}
\newcommand{\beas}{\begin{eqnarray*}}
\newcommand{\eeas}{\end{eqnarray*}}

\DeclareMathOperator*{\esssup}{ess\,sup}

\begin{document}

\bibliographystyle{abbrv}
\title {
Deep neural network approximation for
high-dimensional
elliptic PDEs
with boundary conditions
}
\author[$\dagger$,$\ddagger$]{Philipp Grohs}

\author[$\ddagger$]{Lukas Herrmann} 

\affil[$\dagger$]{\footnotesize Faculty of Mathematics, University of Vienna, 
Oskar-Morgenstern-Platz 1,
1090 Vienna, Austria. 
\newline
\texttt{philipp.grohs@univie.ac.at}}

\affil[$\ddagger$]{\footnotesize Johann Radon Institute for Computational and Applied Mathematics, 
Austrian Academy of Sciences, 
Altenbergerstrasse 69,
 4040 Linz,
   Austria.  
   \newline
   \texttt{$\{$philipp.grohs,lukas.herrmann$\}$@ricam.oeaw.ac.at}}
\maketitle
\date{}
\begin{abstract}
In recent work it has been established that deep neural networks are capable of approximating solutions to a large class of parabolic partial differential equations without incurring the curse of dimension. 
However, all this work has been restricted to problems formulated on the whole Euclidean domain. On the other hand, most problems in engineering and the sciences are formulated on finite domains and subjected to boundary conditions.
The present paper considers an important such model problem, namely the Poisson equation on a domain $D\subset \mathbb{R}^d$ subject to Dirichlet boundary conditions.  It is shown that deep neural networks are capable of representing solutions of that problem without incurring the curse of dimension. The proofs are based on a probabilistic representation of the solution to the Poisson equation as well as a suitable sampling method. 
\end{abstract}



\section{Introduction}
The approximation of solutions to partial differential equations (PDEs) 
in high dimensions 
by classical algorithms such as finite difference or finite element methods
is burdened by the so called \emph{curse of dimension}.
This means that the computational cost to achieve a certain accuracy depends exponentially on the dimension 
of the domain with respect to the reciprocal of the accuracy as base.
This is for example improved in the case of so called sparse tensor discretizations. 
There the logarithm of the reciprocal of the accuracy is the base, 
but the dependence with respect to the dimension is still exponential~\cite{vPS_2004}.
This curse of dimension does not appear in Monte Carlo methods, which are stochastic methods and converge in the root mean squared sense. These methods are however typically restricted to 
evaluating the solution of a given PDE at a single point rather than the full computational domain. The approximation of solutions to PDEs in high dimensions on the full computational domain hence remains a challenging problem.

Deep neural networks (DNNs) emerge as an approximation architecture with application in various areas of function approximation theory, 
which are in many cases as good as the established state of the art method, cf.~\cite{BGKP_2019, GHJv18_786,EDGB_2019,OSZ19_839,OPS_2020}.
They are also used in the context of uncertainty quantification
to approximate mappings that result in parametrized physical systems, where each realization is computationally expensive
and treated as an offline cost, cf.~\cite{LMR_2020,HSZ_2020,SZ_2019,KPRS_2019,GPRSK_2020}. 
The weights of DNNs are usually obtained by approximately solving an optimization problem with a given loss functional defined with computed training data, see for example~\cite{LMR_2020,MR_2020,EHJK_2019}. 

Recently,
there has been vivid research in the approximation of solutions to PDEs in high dimensions posed on $\bbR^d$ by DNNs,
cf.~\cite{GHJv18_786,JSW18_788,GJS19_854,GGJKS19_865,BGJ20_889}.
In~\cite{GHJv18_786}, the authors prove that
DNNs are capable to overcome the curse of dimension in the case 
of certain parabolic PDEs posed on all of $\bbR^d$.
In several recent works this ability of DNNs has also been proven for certain other PDEs on all of $\bbR^d$, also including non-linear PDEs, cf.~\cite{JSW18_788,BGJ20_889}.
However, many applications in engineering and in the sciences require the numerical solution of PDEs with boundary conditions.
Therefore, in this work, we seek to numerically approximate solutions to elliptic PDEs in bounded domains with boundary conditions 
such that the curse of dimension can be overcome.

In particular we establish the first result on the approximation of solutions to PDEs with boundary conditions without curse of dimension using DNNs. 
More precisely, we consider the Poisson equation 
\begin{equation*}
 -\Delta u =f \quad\text{in } D
 \quad\text{and}\quad 
 u\big\vert_{\partial D }=g,
\end{equation*}
where $D \subset \bbR^d$ is bounded and convex.
Our main result, Theorem~\ref{thm:main_result}, states that the solution $u$ can be approximated to within accuracy $\delta$ by a DNN of size scaling polynomially in $d$ and $\delta^{-1}$ whenever an analogous approximation property holds for the right hand side $f$, the boundary condition $g$, as well as the distance function $x\mapsto {\rm dist}(x,\partial D)$. Theorem~\ref{thm:main_result} may thus be interpreted as a ``regularity result'' in the sense that the property of being representable by DNNs without the curse of dimension is conserved under the solution operator of the Poisson equation.

We explicitly establish the required DNN approximation property for $x\mapsto {\rm dist}(x,\partial D)$ for $D$ a cube or a Euclidean ball. On the way to this result we derive a novel DNN approximation for the square root function at a spectral rate that may be of independent interest, see Lemma \ref{lem:sqrt_NN}. 
There has been another approach to approximate the Euclidean norm by DNNs based on the observation that the Euclidean norm is a rotation symmetric function, cf.~\cite{MCCANE2018119,OPS_2020}.

Theorem \ref{thm:main_result} is similar in spirit to other existing works~\cite{GHJv18_786,BGKP_2019} where Monte Carlo methods have been used to show existence of the DNN weights, cf.~\cite{GHJv18_786}.  
The approaches and techniques in the presented manuscript differ significantly for the reason that the behavior of the solution near the boundary needs to be taken into account, which complicates the analysis.

The structure of the manuscript is as follows.
In Section~\ref{sec:NNs}, we briefly recapitulate basic facts on on DNNs. 
In Section~\ref{sec:WOS}, 
we introduce the \emph{walk-on-the-sphere} algorithm and prove 
basic properties. It serves as a tool in Section~\ref{sec:PDEs_DNNs}, where
we show the existence of DNNs that approximate the solution to certain elliptic PDEs with boundary conditions.

\section{Neural networks}
\label{sec:NNs}

Let $\sigma:\bbR \to \bbR$ be the rectified linear unit (ReLU) activation function, 
which is defined by $\sigma(x) := \max\{0,x\}$, $x\in \bbR$.
We consider in general fully connected DNNs of depth $L\in \bbN$.
Let $(N_i)_{i=0,\ldots,d}$ be a sequence of positive integers. 
Let $A^i \in \bbR^{N_{i} \times N_{i-1}}$ and $b^i\in \bbR^{N_i}$, 
$i=1,\ldots,L$.
We define the realization of the DNN $\bbR^{N_0} \ni x\mapsto \phi^L(x)$ by 
\begin{equation}\label{eq:def_relu_nn}
\bbR^{N_0}\ni x \mapsto \phi^i(x)
:=
A^i
\sigma(\phi^{i-1}(x)) + b^i, 
\quad i=2,\ldots,L,
\quad 
\text{with}
\quad  
\bbR^{N_0}\ni x \mapsto \phi^1(x)
:= A^1 x + b^1, 
\end{equation}
where $\bbR^{N } \ni x \mapsto \sigma (x):= (\sigma(x_1),\ldots,\sigma(x_N)) $, $N\in\bbN$,
is defined coordinatewise.
The \emph{weights} of the ReLU DNN $\phi^L$ are the entries of $(A^i, b^i)_{i=1,\ldots,L}$. 
The size of the ReLU DNN $\phi^L$ is defined to be the number of 
non-zero weights and will be denoted by ${\rm size}(\phi^L)$. 
The width of the ReLU DNN $\phi^L$ is defined by ${\rm width}(\phi^L) = \max\{N_0,\ldots,N_L\}$
and $L$ is the depth of $\phi^L$.
Sometimes in the literature~\cite{EDGB_2019} it is distinguished between the architecture of the DNN 
and the realization, which is the function that is induced, see for example~\eqref{eq:def_relu_nn}.
In this manuscript, we shall not make this distinction, since we are mostly interested in asymptotic upper bounds of the size of ReLU DNNs. However, we note that the realization does not uniquely determine the weights.
Moreover, also the depth $L$ in the notation of the ReLU DNN $\phi^L$ shall not be made explicit in the following (meaning that we will drop the superscript $L$).
In this manuscript, we only consider DNNs with ReLU 
activation function
and will mostly write ReLU DNN or just DNN.

The following lemma is~\cite[Proposition~3]{yarotsky_2017}.
\begin{lemma}\label{lem:NN_prod_scalars}
Let $c>0$. For every $\delta\in (0,1)$, there exists a DNN $\tilde\times_\delta$ 
such that 
\begin{equation*}
 \sup_{a,b\in [-c,c]}|ab - \tilde\times_\delta(a,b)|
 \leq \delta
\end{equation*}
and ${\rm size}(\tilde\times_\delta) = \calO(\lceil\log(\delta^{-1})\rceil)$.
\end{lemma}
The following two lemmas are versions of~\cite[Lemmas~II.5 and~II.7]{EDGB_2019}.
\begin{lemma}\label{lem:composition_NNs}
Let $\phi_1$ and $\phi_2$ be two ReLU DNNs 
with input dimensions $N_0^{i}\in \bbN$ and output dimensions $N_L^{i}\in \bbN$, 
$i=1,2$, such that $N_0^1 = N_L^2$.
There exists a ReLU DNN $\phi$ such that 
${\rm size}(\phi ) = 2 {\rm size}(\phi_1 ) + 2{\rm size}(\phi_2 )  $ 
and $\phi(x) = \phi_1(\phi_2(x))$ for every $x\in\bbR^{N_0^2}$.
\end{lemma}
\begin{lemma}\label{lem:additions_NNs}
Let $\phi_{i}$, $i=1,\ldots, n$, be ReLU DNNs with the same input dimension $N_0\in\bbN$, ${\rm size}(\phi_i)=M_i$, ${\rm width}(\phi_i) = W_i$ and depths $L_i $, $i=1,\ldots,n$.
Let $a_i$, $i=1,\ldots,n$, be scalars.
Then there exists a ReLU DNN $\phi$
with ${\rm size}(\phi) = \sum_{i=1}^n  [M_i + W_i + 2(L - L_i) +1] $
such that 
$\phi(x) =\sum_{i=1}^n a_i \phi_i(x)$ for every $x\in \bbR^{N_0}$,
where $L = \max\{ L_1 ,\ldots, L_n\}$.
\end{lemma}
Throughout this manuscript, we will construct ReLU DNNs mostly by composition and addition of already existing ReLU DNNs. The asserted upper bounds in later parts of the manuscript on the size of certain DNNs then result by
Lemmas~\ref{lem:composition_NNs} and~\ref{lem:additions_NNs}.
 
\section{Basics on the walk-on-the-sphere algorithm} 
\label{sec:WOS}

In this work we consider the following elliptic PDE with Dirichlet boundary conditions, 
\begin{equation}\label{eq:elliptic_PDE}
 -\Delta u =f \quad\text{in } D
 \quad\text{and}\quad 
 u\big\vert_{\partial D }=g.  
\end{equation}
Here, $D\subset \bbR^d$ is a convex, bounded domain and $f$ and $g$ are continuous functions.
To this end,
for any $x_0\in \bbR^d$ and $c>0$
$B(x_0,c) := \{x\in \bbR^d : |x-x_0|\leq c\}$. 
Let $C^{k}(\overline{D})$ denote the $k$-times continuously differentiable functions on the closure $\overline{D}$, $k\in\bbN$. 
Denote by $L^\infty(D)$ the space of essentially bounded functions on $D$ with the usual norm 
$v\mapsto \|v\|_{L^{\infty}(D)}:=\esssup_{x\in D}|v(x)|$.
The volume of the domain $D$ with respect to the Lebesgue measure is denoted by $|D|$
and the diameter is denoted by ${\rm diam}(D)$.
The Euclidean norm on $\bbR^d$ will be denoted 
by $x\mapsto |x|$.
For any measured space $(B,\calB,\nu)$ and separable Hilbert space $F$
we denote the Hilbert space of $F$-valued square integrable functions
by $L^2(B,\nu;F)$.
In the case that $\nu$ is the Lebesgue measure and $F$ the real numbers, we shall simply write $L^2(B)$.

We shall recall some elements from the theory of Brownian motion.
Let $(\Omega, \bbF, \calA, \bbP )$ be a filtered probability space and let 
$W_t$, $t\geq 0$, denote a $d$-dimensional Brownian motion, which is adapted to the filtration 
$\bbF=(\calF_t)_{t\geq 0}$, i.e., $W_t$ is $\calF_t$-measurable. 
Define the stochastic process 
\begin{equation*}
 X_t = X_0 + W_t, 
 \quad 
 t\geq 0,
\end{equation*}
where $X_0$ is $\calF_0$-measurable. 
Further, let us denote by $\bbP_x$ the probability measure conditioned on $X_0 =x$
for every $x\in \bbR^d$. 
The expectation with respect to $\bbP_x$ will be denoted by $\bbE_x(\cdot)$.

For any open, non-empty set $\tilde{D}\subset \bbR^d$, define the first exit time of the process $X_t$ starting at $x\in \tilde{D}$ from $\tilde{D}$
by
\begin{equation*}
\tau_{\tilde{D}} :=
\inf\{t\geq 0 : X_t\notin \tilde{D}\}.
\end{equation*} 
Furthermore, 
for any $\varepsilon\in (0,1)$ define the subdomain $D_\varepsilon$ of $D$
by
\begin{equation}\label{eq:def_D_eps}
D_\varepsilon : = \{ x\in D : {\rm dist}(x,\partial D ) > \varepsilon  \}
.
\end{equation}

\begin{lemma}\label{lem:tau_r_ball_exp}
For any $r>0$ and $x\in B(0,r)\subset \bbR^d$, 
it holds that 
\begin{equation*}
\bbE_x(\tau_{B(0,r)})
=
\frac{r^2 - |x|^2}{d}
\end{equation*}
\end{lemma}
\begin{proof}
This is explicitly~\cite[Proposition~3.1.8]{BM_Port_1978}.
\end{proof}

\begin{lemma}\label{lem:expectation_tau-tau_eps}
Let $D\subset \bbR^d$ be a bounded convex domain. 
For every $\varepsilon \in (0,1)$ such that $D_\varepsilon\subset D$ defined in~\eqref{eq:def_D_eps}is not empty
and for every $x\in D_\varepsilon$,
\begin{equation*}
\bbE_x(\tau_D - \tau_{D_\varepsilon})
\leq 
{\rm diam}(D) \varepsilon
.
\end{equation*}
\end{lemma}

\begin{proof}
By the tower property,
\begin{equation*}
\bbE_x(\tau_D - \tau_{D_\varepsilon})
=
\bbE_x(\bbE_{X_{\tau_{D_\varepsilon}}}(\tau_D - \tau_{D_\varepsilon}))
=
\bbE_x(\bbE_{X_{\tau_{D_\varepsilon}}}(\tau_D ))
.
\end{equation*}
Note that $X_t$, $t\geq 0$, under the measure $\bbP_{X_{\tau_{D_{\varepsilon}}}}$
has the same distribution as $X_{\tau_{D_\varepsilon}} + \widetilde{W}_{t}$, $t\geq 0$, 
where $\widetilde{W}_t$, $t\geq 0$, is a Brownian motion that is indenpendent from ${W}_t$, $t\geq 0$.
To estimate the expectation of the stopping time conditioned on $X_{\tau_{D_{\varepsilon}}}$, we use that there exist two parallel hyperplanes $A,B$ of dimension $d-1$ which do not intersect with $D$. 
One of the hyperplane satisfies that ${\rm dist}(X_{\tau_{D_{\varepsilon}}},A)=\varepsilon$. 
The other hyperplane $B$ satisfies that ${\rm dist}(X_{\tau_{D_{\varepsilon}}},B)={\rm diam}(D)$. 
Specifically, there exists a unit vector $c\in \bbR^d$ and $\alpha,\beta\in\bbR$
such that $A = \{v = x + \alpha c: x\in\bbR^d ,x^\top c =0\}$
and $B = \{v = x + \beta c: x\in\bbR^d ,x^\top c =0\}$, where $\beta<\alpha$ and $\alpha-\beta = \varepsilon + {\rm diam}(D)$.
They exist due to the assumed convexity of the domain $D$.
Also define the unbounded domain of points in between $A$ and $B$ by 
$\widetilde{D}=\{x\in \bbR^d : x^\top c \in (\beta,\alpha)\}$.
It follows that $\bbE_{X_{\tau_{D_\varepsilon}}}(\tau_D ) \leq \bbE_{X_{\tau_{D_\varepsilon}}}(\tau_{\widetilde{D}} )$.
Under the measure $\bbP_{X_{\tau_{D_{\varepsilon}}}}$, the stopping time $\tau_{\widetilde{D}}$
satisfies that 
$\tau_{\widetilde{D}}
= \inf\{t\geq 0 : X_t^\top c - \alpha \notin (-{\rm diam}(D),\varepsilon)\}$.
Moreover since $c$ is a unit vector, under the measure $\bbP_{X_{\tau_{D_{\varepsilon}}}}$,
the process $X_t^\top c - \alpha$ is a one-dimensional 
Brownian motion that starts at zero.
Thus, 
$\bbE_{X_{\tau_{D_\varepsilon}}}(\tau_{\widetilde{D}} ) 
={\rm diam}(D)\varepsilon$.
See Figure~\ref{fig:acc_WOS} for a geometric illustration.

\begin{figure}[h]
  \centering \includegraphics[width=0.85\textwidth]{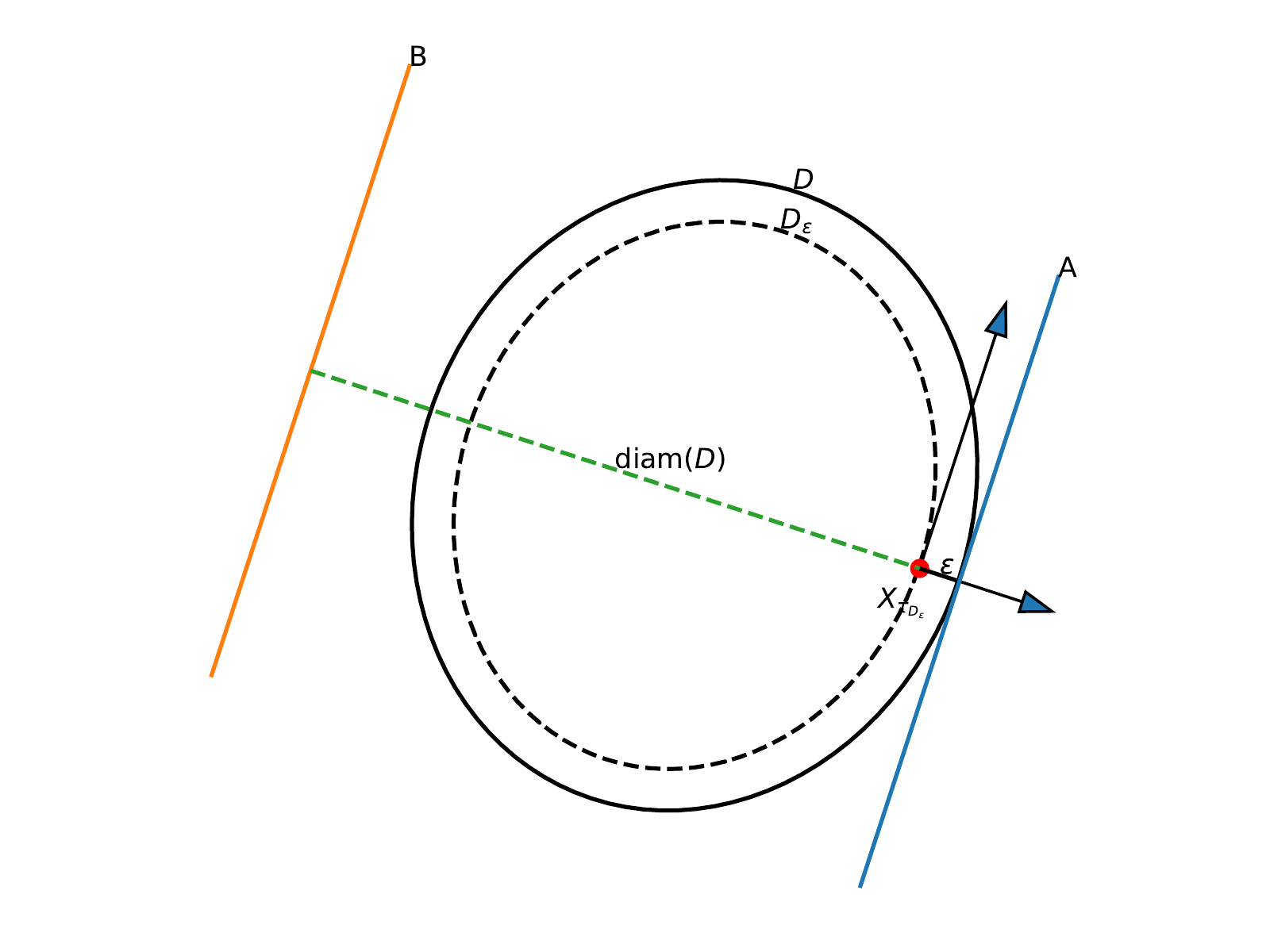}
  \caption{Illustration of a realization of $X_{\tau_{D_\varepsilon}}$. Here $D$ is given by an ellipse.
  The domain $D$ may be positioned in between two hyperplanes $A$ and $B$ such that 
  ${\rm dist}(X_{\tau_{D_\varepsilon}},\partial D )={\rm dist}(X_{\tau_{D_\varepsilon}},A )=\varepsilon$.}
  \label{fig:acc_WOS}
\end{figure}
We recall the fact that for a one dimensional Brownian motion starting at zero, the expected time 
such that it leaves the interval $[a,b]$ for $a<0<b$ is equal to $|ab|$, 
see~\cite[Proposition~2.2.20]{BM_Port_1978}.
Thus the claimed estimate follows. 
\end{proof}

The following result is also implied by~\cite[Theorems~9.13 and~9.17]{Oeksendal}. 
We give a proof to establish some techniques to be used throughout this section.

\begin{proposition}\label{prop:Feynman--Kac_repr}
Suppose that $D\subset \bbR^d$ is convex and bounded.
Let $f$ be H\"older continuous and let $g$ be extendable to $\overline{D}$ such that $g\in C^2(\overline{D})$.
Then, for every $x\in \overline{D}$,
\begin{equation*}
u(x) 
=
\bbE_x(g(X_{\tau_D}))
+
\frac{1}{2}
\bbE_x\left( \int_{0}^{\tau_D}f(X_s){\rm d}s\right).
\end{equation*}
\end{proposition}

\begin{proof}
We will first establish a formula of the type asserted in this proposition in the interior of $D$ and then 
extend it to also incorporate the boundary.

The assumed convexity of the domain $D$ implies that all boundary points are \emph{regular} in the
sense of~\cite{GilbargTrudinger};
for details see \cite[pp.~25, 27]{GilbargTrudinger}. 
In conjunction with~\cite[Theorem~4.3]{GilbargTrudinger}, it follows that
the solution $u$ to \eqref{eq:elliptic_PDE} exists and is unique.
More precisely by~\cite[Lemma~4.2]{GilbargTrudinger}, $u$ is twice continuously differentiable in $D$. 

Let $\varepsilon >0$ be arbitrary. 
Recall $D_\varepsilon : = \{ x\in D : {\rm dist}(x,\partial D ) > \varepsilon  \}$. 
Suppose that $\varepsilon$ is sufficiently small such that $D_\varepsilon$ 
is not empty.
By Ito's formula (see for example~\cite[Theorem~17.8]{Schilling_2014}), for every $x\in D_{\varepsilon}$
\begin{equation*}
u(X_{\tau_{D_\varepsilon}}) -  u(x) 
= \int_{0}^{\tau_{D_\varepsilon}} \nabla u(X_t) \cdot {\rm d}W_t 
+ \frac{1}{2}\int_{0}^{\tau_{D_\varepsilon}} \Delta u(X_t) {\rm d}t. 
\end{equation*}
The optional stopping theorem (see for example~\cite[Theorem~A.18 and Remark~A.21]{Schilling_2014})
implies that for an integrable martingale $M_t$, $t\geq 0$, and an integrable stopping time $\widetilde{\tau}$, 
it holds that $\bbE(M_{\widetilde{\tau}}) = \bbE(M_0)$.
As a consequence, the martingale property of the stochastic integral $\int_{0}^{t} \nabla u(X_s) \cdot {\rm d}W_s , t\geq 0$, see~\cite[Proposition~3.2.10]{KaratzasShreve_1988},
implies
\begin{equation}\label{eq:exp_stopped_stoch_integral}
\bbE_x\left( \int_{0}^{\tau_{D_\varepsilon}} \nabla u(X_t) \cdot {\rm d}W_t\right)=0
.
\end{equation}
and thus
\begin{equation*}
u(x) 
= 
\bbE_x(u(X_{\tau_{D_\varepsilon}}))
+
\frac{1}{2} \bbE_x\left(  
\int_{0}^{\tau_{D_\varepsilon}} f(X_t) {\rm d}t
\right ).
\end{equation*}
We seek to study the limit $\varepsilon \to 0$.
The solution $u$ is Lipschitz continuous on the closure $\overline{D}$ with Lipschitz constant $L_u>0$, 
which may be concluded by~\cite[Theorem~1.4]{CM_2011}.
There, the statement of~\cite[Theorem~1.4]{CM_2011} is applied to $v= u -g $ with right hand side $f-\Delta g$. 
The Lipschitz continuity of $u$ yields
\begin{equation*}
|\bbE_x(u(X_{\tau_D}) - u(X_{\tau_{D_\varepsilon}}))|
\leq 
L_u
\bbE_x(|X_{\tau_D} - X_{\tau_{D_\varepsilon}}|)
.
\end{equation*}
Since for any two integrable stopping times $\tilde\tau$, $\bar\tau$ that satisfiy
$\tilde\tau \geq \bar\tau$ it holds 
\begin{equation}\label{eq:mean_sq_cont_stop_time}
\bbE_x(|X_{\tilde\tau} - X_{\bar\tau}|^2)  = \bbE(\tilde\tau - \bar\tau), 
\end{equation}
cf.~\cite[Corollary~2.46 and Theorem~2.48]{MP_2010},
we obtain by Lemma~\ref{lem:expectation_tau-tau_eps}
\begin{equation*}
\bbE_x(|X_{\tau_D} - X_{\tau_{D_\varepsilon}}|)
\leq 
\sqrt{\bbE_x(|X_{\tau_D} - X_{\tau_{D_\varepsilon}}|^2)}
=
\sqrt{ \bbE_x(\tau_D - {\tau_{D_\varepsilon}}) }
\leq 
\sqrt{{\rm diam}(D) \varepsilon }
,
\end{equation*}
which implies that $ \bbE_x(u(X_{\tau_{D_\varepsilon}})) \to \bbE_x( g(X_{\tau_{D}}))$
as $\varepsilon \to 0$.
Similarly, also by Lemma~\ref{lem:expectation_tau-tau_eps}
\begin{equation*}
\left|\bbE_x\left( \int_{\tau_{D_\varepsilon}}^{\tau_D} f(X_t){\rm d}t\right)\right|
\leq 
\|f\|_{L^\infty(D)}\bbE_x(\tau_D - \tau_{D_\varepsilon})
\leq
\|f\|_{L^\infty(D)} {\rm diam}(D) \varepsilon  
\end{equation*}
and thus $\bbE_x(\int_0^{\tau_{D_\varepsilon}} f(X_t){\rm d}t )  \to \bbE_x(\int_0^{{\tau_{D}}} f(X_t){\rm d}t ) $
as $\varepsilon \to 0$.
\end{proof}

We define the discrete processes $\bar{X}_k$, $k\geq 0$, 
and $r_k$, $k\geq 1$,
which also tacitly depend on an initial starting point $x\in D$.
Let $\bar{X}_0=x$ and 
\begin{equation}\label{eq:def_Xbar_k} 
\bar{X}_k 
= 
\bar{X}_{k-1} 
+
Y_k r_k
\quad \text{with}\quad
r_k 
=
{\rm dist }(\bar{X}_{k-1},\partial D), 
\quad 
k\geq 1 .
\end{equation}
The sequence $Y_k$, $k\geq 1$, 
is indenpendent and identically distributed according
to $X_{\tau_{B(0,1)}}$ with respect to $\bbP_0$.
Note that $Y_k$ is uniformly distributed on the unit sphere, $k\geq 1$, see~\cite[Theorem~2]{Wendel_1980}.
This process has been introduced in~\cite{Muller_1956} and is commonly referred to as \emph{walk-on-the-sphere}.
The resulting random vector $\bar{X}_k$
depends on the initial point $\bar{X}_0=x$ and the random directions 
$Y_{k'}$, $k'=1,\ldots,k$.
Sometimes, we shall use the notation
\begin{equation}\label{eq:Xbar_explicit_notation}
\bar{X}_k = \bar{X}_k(x,Y_1,\ldots,Y_k), 
\quad k\geq 0,
\end{equation}
where this dependence is explicit.

The process $\bar{X}_k$, $k\geq 0$, is related to $X_t$, $t\geq 0$ as follows. 
Define $\bar{r}_1 := {\rm dist}(x,\partial D)$
and $\tau_1:= \tau_{B(x,r_1)}$.
For every $k\geq 2$, 
\begin{equation}\label{eq:def_tau_k}
\bar{r}_k := {\rm dist}(X_{\tau_{k-1}},\partial D)
\quad\text{and}\quad 
\tau_k :=
\inf\{t\geq 0 : X_{t+\tau_{k-1}} \notin B(X_{\sum_{j=1}^{k-1}\tau_{j}},\bar{r}_k) \}
.
\end{equation}
Note that $r_k$ and  $\bar{r}_k$ have the same distribution, $k\geq 1$.
Define 
\begin{equation}\label{eq:def_calI}
\calI(k) := \sum_{j=1}^{k}\tau_{j}
\quad\forall k\in\bbN.
\end{equation}
As noted above, under the measure $\bbP_{X_{\calI(k-1)}}$, $X_{\calI(k)}$ is equally distributed on the boundary of the ball $B(X_{\calI(k-1)},\bar{r}_k)$. 
Thus, by construction of the process $\bar{X}_k$, $k\geq 0$, 
it holds that 
$\bar{X}_k$ has the same distribution as $X_{\calI(k)}$, $k\geq 1$, 
see also Figure~\ref{fig:WOS_illustration}.
\begin{figure}[h]
  \centering \includegraphics[width=0.85\textwidth]{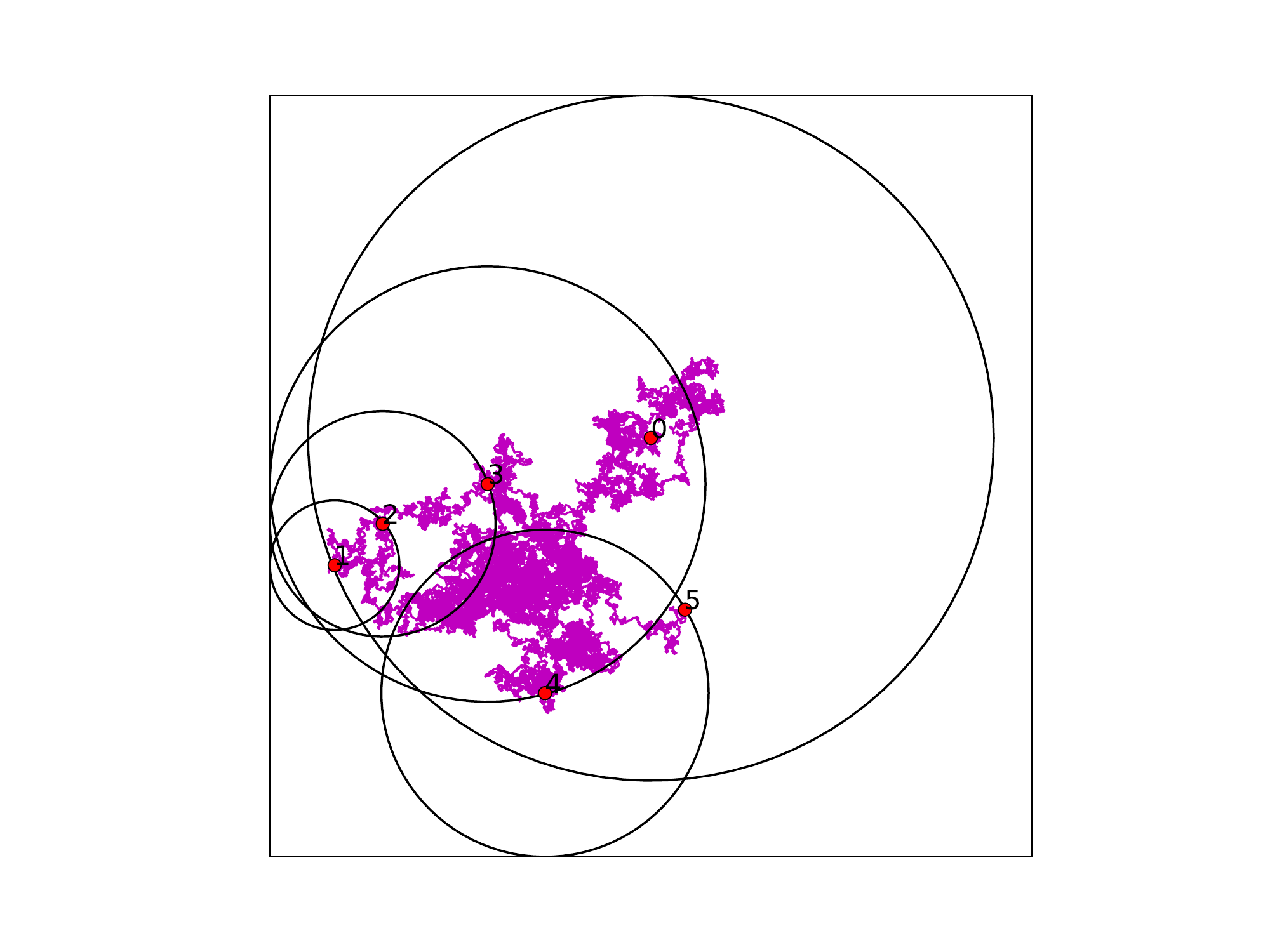}
  \caption{Illustration of the process $\bar{X}_k$, $k\geq 0$,
  in the case of a rectangular domain. One realization of $X_t$, $t\in [0,\calI(5)]$,
  and $\bar{X}_k$, $k=0,\ldots,5$, is plotted as numbered annotations.}
  \label{fig:WOS_illustration}
\end{figure}
Note that $\lim_{k \to \infty} X_{\calI(k)} = X_{\tau_D}$ $\bbP$-a.s., cf.~\cite[Theorem~3.6]{Muller_1956}.
The following lemma is a version of~\cite[Lemma~6.3]{KOS_2018} for $\alpha=2$ (in the notation of~\cite{KOS_2018}). 
We give a proof for the convenience of the reader.

\begin{lemma}\label{lem:sum_repr_Poisson_part}
Let the assumptions of Proposition~\ref{prop:Feynman--Kac_repr}
be satisfied.
There holds that
\begin{equation*}
u(x)
=
\bbE_x(g(X_{\tau_D}))
+
\bbE
\left( 
\sum_{k\geq 1}
r_{k}^2
K_1(f(\bar{X}_{k-1} + r_{k} \cdot))
\right),
\end{equation*}
where for any continuous function $v:B(0,1)\to \bbR$
\begin{equation}\label{eq:def_K1}
K_1(v) 
:=
\frac{1}{2}
\bbE_{0}\left( 
\int_{0}^{\tau_{B(0,1)}}
v(X_t){\rm d}t
\right)
.
\end{equation}
\end{lemma}

\begin{proof}
This proof builds on the representation of $u$ from Proposition~\ref{prop:Feynman--Kac_repr}.
In particular, the term $\bbE_x(\int_0^{\tau_D} f(X_t) {\rm d}t )$
shall be represented by a sum of consecutive solutions of the Poisson equation on a ball.
 
The strong Markov property of the Brownian motion yields
\begin{equation}\label{eq:term1}
\bbE_x\left(
\int_{0}^{\tau_D}
f(X_t){\rm d}t
\right)
=
\bbE_x
\left(
\sum_{k\geq 1}
\int_{\tau_{k-1}}^{\tau_k}
f(X_t){\rm d}t
\right)
=
\bbE_x\left( 
\sum_{k\geq 1}
\bbE_{X_{\calI(k-1)}}\left( 
\int_{0}^{\tau_k}
f(X_t){\rm d}t
\right)
\right),
\end{equation}
where we recall the definition of the stopping times $\tau_k$, $k\geq 1$, from~\eqref{eq:def_tau_k}
and set $\tau_0=0$. 
Note that the \emph{scaling property} of Brownian motion states that 
$c X_{c^{-2} t}$ has the same distribution as $X_t$ 
for every $c>0$ and every $t\geq 0$, conditioned on $X_0 = x$, $x\in\bbR^d$. 
In conjunction with the strong Markov property for every $x'\in D$
and every $r>0$ such that $B(x',r)\subset D$,
\begin{equation*}
\bbE_{x'}\left( 
\int_{0}^{\tau_{B(x',r)}}
f(X_t){\rm d}t
\right)
=
\bbE_{0}\left( 
\int_{0}^{\tau_{B(0,r)}}
f(x'+X_t){\rm d}t
\right)
=
r^2
\bbE_{0}\left( 
\int_{0}^{\tau_{B(0,1)}}
f(x'+rX_t){\rm d}t
\right)
.
\end{equation*}
The assertion follows by inserting the previous equality into~\eqref{eq:term1}
with $x' = \bar{X}_{k-1}$ 
and $r=r_k$ and Proposition~\ref{prop:Feynman--Kac_repr}
using that $\bar{X}_{k-1}$ and $X_{\calI(k-1)}$ 
have the same distribution.
\end{proof}
Note that the functional $v\mapsto K_1(v)$ denotes the solution to the Poisson equation with homogeneous Dirichlet boundary conditions on the unit ball evaluated at the origin.
An explicit formula for $K_1$ is a classical result by Boggio~\cite{boggio_1905}, see also~\cite{GGS_2010}. 
Specifically by \cite[Lemma~2.27]{GGS_2010} for $d\geq 3$,
\begin{equation}\label{eq:boggios_formula}
K_1(v) 
= 
\frac{\Gamma(1+d/2)}{d \pi^{d/2}}\frac{1}{d-2}
\int_{|y|\leq 1}
v(y)
(|y|^{2-d} -1){\rm d}y
.
\end{equation}

For every $\varepsilon>0$, let us define the random index $N(\varepsilon)$
by
\begin{equation*}
N(\varepsilon)
:=
\inf\{k\geq 1 : {\rm dist}(\bar{X}_k,\partial D)\leq \varepsilon \}
.
\end{equation*}

\begin{proposition}\label{prop:error_of_WOS}
Let the assumptions of Proposition~\ref{prop:Feynman--Kac_repr}
be satisfied.
Let $\varepsilon>0$ be such that $D_\varepsilon \subset D$ is non empty.
Let $\bar{N}$ be an integer valued random variable that satisfies $\bar{N} \geq N(\varepsilon)$ $\bbP$-a.s.
It holds that
\begin{equation}\label{eq:err_WOS_Poisson_part}
\left|
\bbE_x\left(\int_{0}^{\tau_D}f(X_s){\rm d}s \right)
-
\bbE\left( \sum_{k=1}^{\bar{N}} r_{k}^2 K_1(f(\bar{X}_{k-1} + r_{k}\cdot)) \right)
\right|
\leq 
{\rm diam}(D) \|f\|_{L^\infty(D)}
\varepsilon
\end{equation}
and
\begin{equation}\label{eq:err_WOS_Laplace_part}
\left | 
\bbE_x(g(X_{\tau_D}))
-
\bbE_x(g(\bar{X}_{\bar{N}}))
\right|
\leq 
\frac{1}{2}\|\Delta g \|_{L^\infty(D)}
{\rm diam}(D)
\varepsilon.
\end{equation}
\end{proposition}

\begin{proof}
 As a consequence of Lemma~\ref{lem:sum_repr_Poisson_part}
 and its proof, see~\eqref{eq:term1},
 \begin{equation*}
 \bbE_x\left(\int_{0}^{\tau_D} f(X_s){\rm d}s
 \right)
 -
 \bbE\left(
 \sum_{k=1}^{ \bar{N}} r_{k}^2 K_1(f(\bar{X}_{k-1} + r_{k}\cdot) )
 \right)
 = 
  \bbE_x\left(\int_{\calI(\bar{N})}^{\tau_D} f(X_s){\rm d}s\right)
 \end{equation*}
The assumption $ \bar{N} \geq N(\varepsilon)$, $\bbP$-a.s.,
implies that $\calI( \bar{N}) \geq \tau_{D_\varepsilon}$, $\bbP$-a.s.,
where we recall the defintion of $\calI$ in~\eqref{eq:def_calI}.
Thus, 
\begin{equation*}
\left|
\bbE_x\left(\int_{0}^{\tau_D} f(X_s){\rm d}s
\right)
-
\bbE\left(
\sum_{k=1}^{ \bar{N}} r_{k}^2 K_1(f(\bar{X}_{k-1} + r_{k}\cdot) ) 
\right)
\right |
\leq 
\|f\|_{L^\infty(D)}
\bbE_x(\tau_D - \tau_{D_\varepsilon})
.
\end{equation*}
The first assertion~\eqref{eq:err_WOS_Poisson_part} follows by Lemma~\ref{lem:expectation_tau-tau_eps}.
To show the second assertion~\eqref{eq:err_WOS_Laplace_part},
we may apply Ito's lemma (see for example~\cite[Theorem~17.8]{Schilling_2014}), which implies
\begin{equation*}
g(X_{\tau_D}) - g(X_{\calI(\bar{N})})
=
\int_{\calI(\bar{N})}^{\tau_D}
\nabla g(X_s) \cdot {\rm d} W_s
+
\frac{1}{2}
\int_{\calI(\bar{N})}^{\tau_D}
\Delta g(X_s) {\rm d}s
.
\end{equation*}
By~\eqref{eq:exp_stopped_stoch_integral}, 
\begin{equation*}
\bbE_x\left(
\int_{\calI(\bar{N})}^{\tau_D}
\nabla g(X_s) \cdot {\rm d} W_s
\right)
=0.
\end{equation*}
The assertion follows by 
\begin{equation*}
\left|\bbE_x\left(
\int_{\calI(\bar{N})}^{\tau_D}
\Delta g(X_s) {\rm d}s
\right)\right|
\leq 
\|\Delta g\|_{L^{\infty}((D)}
\bbE_x(\tau_D - \tau_{D_\varepsilon})
\leq 
\|\Delta g\|_{L^{\infty}(D)}
{\rm diam}(D)\varepsilon
,
\end{equation*}
which is consequence of Lemma~\ref{lem:expectation_tau-tau_eps}.
\end{proof}

The statement of the following lemma is in principle known. We provide a proof for the convenience of the reader.

\begin{lemma}\label{lem:exp_of_N}
Let $D\subset \bbR^d$ be a bounded, convex domain.
For any $\varepsilon>0$ such that $D_\varepsilon$ 
is non-empty,
it holds that
\begin{equation*}
\bbE\left( \sup_{x\in D} N(\varepsilon ) \right)
\leq 
({\rm diam}(D))^2 \varepsilon^{-2}
.
\end{equation*}
\end{lemma}

\begin{proof}
We recall that for $x\in D$ such that ${\rm dist}(x,\partial D)<\varepsilon$, 
$N(\varepsilon ) =1$.
Thus, by the Markov property and Lemma~\ref{lem:tau_r_ball_exp}, for every $k\geq 1$,
\begin{equation*}
\bbE_{X_{\calI(k-1)}}(\tau_{k})
=
\frac{r_k^2}{d}
\quad \bbP\text{-a.s.}, 
\end{equation*}
where $\tau_k$ is defined in~\eqref{eq:def_tau_k}.
Note that since $D_\varepsilon$ is not empty,
it follows that $\varepsilon\leq {\rm diam}(D)/2$.
Since $r_k\geq \varepsilon$ for every $k\leq N(\varepsilon)$ and $x\in D_\varepsilon$ $\bbP$-a.s.,
\begin{equation*}
\begin{aligned}
\bbE\left( \sup_{x\in D} N(\varepsilon ) \right)
\leq
\bbE\left( 
\sup_{x\in D_\varepsilon}
\sum_{k\geq 1}
\mathbbm{1}_{\{r_{k}\geq \varepsilon\}}
\right)
&\leq 
d
\varepsilon^{-2}
\bbE\left( 
\sup_{x\in D}
\sum_{k\geq 1}
\bbE_{X_{\calI(k-1)}}(\tau_{k})
\mathbbm{1}_{\{r_k\geq \varepsilon\}}
\right)
\\
&\leq 
d
\varepsilon^{-2}
\bbE\left(
\sup_{x\in D}
\tau_D
\right)
.
\end{aligned}
\end{equation*}
Since $\{t > 0 : W_t \notin B(0,{\rm diam}(D)) \} \subset \{t>0: x+W_t \notin D\}$
for every $x\in D$ (using convexity of $D$), it holds that 
$\sup_{x\in D}\tau_D \leq  \inf\{t>0 :W_t \notin B(0,{\rm diam}(D)) \}$.
By Lemma~\ref{lem:tau_r_ball_exp}, 
\begin{equation*}
\bbE\left(
\sup_{x\in D}
\tau_D
\right) 
\leq 
\frac{({\rm diam}(D))^2}{d}, 
\end{equation*} 
which concludes the proof of this lemma.
\end{proof}

\section{Approximation by deep neural networks without curse of dimension}
\label{sec:PDEs_DNNs}

Deep neural networks allow to accommodate composition of mappings in their structure or architecture.
The repeated occurrence of linear maps (here expectations such as $\bbE_x(\cdot)$) and compositions of maps in the Feynman--Kac representation of a solution to an elliptic or also to a parabolic PDE
was found to suit the architecture of deep neural networks in~\cite[Proposition~3.4]{GHJv18_786}.
In this section the applicability of DNNs shall be extended to PDEs with boundary conditions.
The main obstruction in the analysis is the stopping time $\tau_D$, 
which also depends on $x\in D$; the point where the process $X_t$ is started.

Recall that we aim to approximate solutions to the prototypical elliptic PDE 
\begin{equation} \tag{{\ref{eq:elliptic_PDE}}}
 -\Delta u =f \quad\text{in } D
 \quad\text{and}\quad 
 u\big\vert_{\partial D }=g 
\end{equation}
by DNNs with ReLU activation function. 
The basics on stochastic sampling methods introduced in Section~\ref{sec:WOS}
shall serve as tools in the proofs of this section. 
The following theorem constitutes our main result. 

\begin{theorem}\label{thm:main_result}
        Let the assumptions of Proposition~\ref{prop:Feynman--Kac_repr}
        be satisfied.
        Suppose that
        for every $\delta, \delta_f, \delta_g\in (0,1)$, there exist ReLU DNNs $\phi_{{\rm dist},\delta}$, $\phi_{f,\delta_f}$, and $\phi_{g,\delta_g}$ such that
\begin{equation}
\sup_{x\in D}|{\rm dist}(x,\partial D) -  \phi_{{\rm dist},\delta}(x) |\leq \delta
\end{equation}
\begin{equation}\label{eq:RHS_NN}
\| f -  \phi_{f,\delta_f}\|_{L^\infty(D)}\leq \delta_f
\end{equation}
\begin{equation}\label{eq:BCs_NN}
\|g -  \phi_{g,\delta_g}\|_{L^\infty(D)}\leq \delta_g
\end{equation}
and ${\rm size}(\phi_{{\rm dist}, \delta}) = \calO(d^a \lceil\log(\delta^{-1})\rceil^b)$, 
${\rm size}(\phi_{f, \delta_f}) = \calO(d^a\delta_f^{-b})$,
and ${\rm size}(\phi_{g, \delta_g}) = \calO(d^a\delta_g^{-b})$
for some $a,b\in (1,\infty)$ which do not depend on $d$.
Let additionally $f$ and $g$ be Lipschitz continuous on $\overline{D}$.
	For every $\bar{\delta}\in (0,1)$, there exists a ReLU DNN $\phi_{u,\bar{\delta}}$
	such that
	\begin{equation*}
	\| u - \phi_{u,\bar{\delta}}\|_{L^2(D)}
	\leq  
	\bar{\delta}
	\end{equation*}
	with ${\rm size}(\phi_{u,\bar{\delta}}) =\calO( d^a\bar{\delta}^{-14-6b} (1+|D|^{14+6b}) )$. 
	The tacit constants in the Landau symbols depend on
$\|f\|_{L^\infty(D)}$, $\|\Delta g\|_{L^\infty(D)}$, $\|g\|_{L^\infty(D)}$, the Lipschitz constants of $f$ and $g$, and on ${\rm diam}(D)$.
\end{theorem}

The proof of Theorem~\ref{thm:main_result}
will be postponed to the end of this section after two intermediate propositions have been proven.

\begin{proposition}\label{prop:NN_approx_Laplace_part}
	Suppose that $g\in C^2(\overline{D})$, $x\mapsto {\rm dist}(x,\partial D)$
	can be realized by a ReLU DNN $\phi_{\rm dist}$, 
	and for any $\delta_g\in (0,1)$
	there exists a ReLU DNN $\phi_{g,\delta_g}$ 
	such that
	\begin{equation*}
	\| g - \phi_{g,\delta_g}\|_{L^\infty(D)}
	\leq
	\delta_g
	.
	\end{equation*}
	For every $\bar{\delta}\in (0,1)$, there exists a ReLU DNN $\phi_{1,\bar{\delta}}$
	such that
	\begin{equation*}
	\sqrt{\int_D |\bbE_{x}(g(X_{\tau_D})) - \phi_{1,\bar{\delta}}(x)|^2{\rm d}x}
	\leq  
	\bar{\delta}
	.
	\end{equation*}
	Furthermore, there exist $M= \lceil  c\bar{\delta}^{-2}(1+|D|)\rceil $, $\bar{N}_i$, $i=1,\ldots,M$, and unit vectors $Y_{i,k}$, $k=1,\ldots,\bar{N}_i$, $i=1,\ldots,M$, such that for every $x \in D$, 
	\begin{equation*}
	\phi_{1,\bar{\delta}}(x)
	=
	\frac{1}{M}\sum_{i=1}^M(\phi_{g,\delta_g}(\bar{X}_{\bar{N}_i}(x,Y_{i,1},\ldots,Y_{i,\bar{N}_i}))).
	\end{equation*}
	The accuracy $\delta_g$ of the ReLU DNN $\phi_{g,\delta_g}$
	satisfies $\delta_g = c' \bar{\delta}/(1 + \sqrt{|D|})$.
	The numbers $\bar{N}_i$, $i=1,\ldots,M$, satisfy that
	\begin{equation}\label{eq:est_N_i_Laplace_part}
	\sum_{i=1}^{M}
	\bar{N}_i
	\leq 
	C M^2 ( M  + |D|)
	.
	\end{equation}
The constants $c,c',C>0$ only depend on $\|\Delta g \|_{L^\infty(D)}$ and on ${\rm diam}(D)$.
\end{proposition}

\begin{proof}
Let $\varepsilon\in (0,1)$  and $\delta_g \in (0,1)$ be arbitrary such that $D_\varepsilon$ is not empty, which will be determined in the following.
Define the random variable $\bar{N}(\varepsilon) := \sup_{x\in D} N(\varepsilon)$.
 By Proposition~\ref{prop:error_of_WOS}, 
 \begin{equation}\label{eq:err1}
  \sup_{x\in D}
  |\bbE_x(g(X_{\tau_D})) - \bbE(g(\bar{X}_{\bar{N}(\varepsilon)}))|
  \leq 
  \frac{1}{2} \|\Delta g\|_{L^\infty(D)} {\rm diam}(D) \varepsilon.
 \end{equation}
 The assumed approximability of $g$ by the ReLU DNN $\phi_{g,\delta_g}$ results in
 \begin{equation}\label{eq:err2}
  \sup_{x\in D}|\bbE(g(\bar{X}_{\bar{N}(\varepsilon)})) - \bbE(\phi_{g,\delta_g}(\bar{X}_{\bar{N}(\varepsilon)}))  |
  \leq 
  \delta_g 
  .
 \end{equation}
Denote by $E_M(\cdot)$, $M\in\bbN$, a Monte Carlo estimator with respect to the 
random variables $\bar{N}(\varepsilon)$
and the sequence $\bar{X}_k$, $k\geq 0$, 
i.e., for every square integrable function $\varphi: D^{\bbN_0} \times \bbN\to\bbR$
\begin{equation}\label{eq:def_MC_est1}
E_M(\varphi)
:=
\frac{1}{M}
\sum_{i=1}^M
\varphi((\bar{X}_k)^{(i)}_{k\geq 0}, \bar{N}^{(i)}(\varepsilon))
,
\end{equation}
where $((\bar{X}_k)^{(i)}_{k\geq 0}, \bar{N}^{(i)}(\varepsilon) )$, $i=1,\ldots,M$,
are mutually independent and have the same distribution as
$((\bar{X}_k)_{k\geq 0}, \bar{N}(\varepsilon) )$.
Recall that $\bar{X}_0 = x$, $x\in D$.
It is well-known that that for 
any square integrable $\varphi$
\begin{equation}\label{eq:MC_err_est}
\begin{aligned}
&\| \bbE(\varphi((\bar{X}_k)_{k\geq 0}, \bar{N}(\varepsilon) )) - E_M(\varphi)\|_{L^2(\Omega, \bbP; L^2(D))}
\\
&\qquad=
\sqrt{\frac{\|\varphi((\bar{X}_k)_{k\geq 0}, \bar{N}(\varepsilon) ) - \bbE(\varphi((\bar{X}_k)_{k\geq 0}, \bar{N}(\varepsilon) ))\|_{L^2(\Omega, \bbP; L^2(D))}}{M}}
.
\end{aligned}
\end{equation}
Thus, by~\eqref{eq:err1}, \eqref{eq:err2}, \eqref{eq:MC_err_est}, and by Lemma~\ref{lem:exp_of_N}
\begin{equation}\label{eq:NN_ext_dist_Laplace_part_err2}
\begin{aligned}
&\bbE\left(  
\int_D | \bbE_x(g(X_{\tau_D}))
- E_M(\phi_{g,\delta_g}(\bar{X}_{\bar{N}(\varepsilon)}))
|^2
{\rm d}x
+ 
\varepsilon^2
\left| \bbE\left(\sqrt{\bar{N}(\varepsilon)}\right) -   E_M\left(\sqrt{\bar{N}(\varepsilon)}\right) \right|^2
\right)
\\
&\qquad\leq
\frac{3}{4}
 \|\Delta g\|_{L^\infty(D)}^2 {\rm diam}(D)^2 |D| \varepsilon^2
 +
3 |D|\delta_g^2
+
3 |D|
\frac{(\|g\|_{L^\infty(D)}  + \delta_g)^2}{M}
+
\frac{{\rm diam}(D)^2}{M}
=: {\rm error}^2_{M,\varepsilon, \delta_g}
\end{aligned}
\end{equation}
The fact that for a positive random variable $Z$ such that $\bbE(Z)\leq c$, 
there exists a set of positive probability $A\subset \Omega$
such that $Z(\omega)\leq c$ for every $\omega\in A$
implies 
 there exist $\bar{N}_{i}$ and direction vectors $Y_{i,k}$, $i=1,\ldots,M$, 
 $k=1,\ldots, \bar{N}_i$, such that
 \begin{equation}\label{eq:Laplace_exct_dist_est1}
 \int_D \left|
 \bbE_x(g(X_{\tau_D}))
- \frac{1}{M}\sum_{i=1}^M(\phi_{g,\delta_g}(\bar{X}_{\bar{N}_i(\varepsilon)}))
\right|^2 {\rm d}x
+ 
\varepsilon^2
\left| \bbE\left(\sqrt{\bar{N}(\varepsilon)}\right) -   \frac{1}{M}\sum_{i=1}^M\left(\sqrt{\bar{N}_i}\right) \right|^2
 \leq 
 {\rm error}^2_{M,\varepsilon,\delta_g}. 
 \end{equation}
 In conjunction with the previous estimate, the Jensen inequality and Lemma~\ref{lem:exp_of_N}, 
 imply
 \begin{equation*}
     \varepsilon^2 
     \left(\frac{1}{M} \sum_{i=1}^M \sqrt{\bar{N}_i}\right)^2
     \leq 
     2
     \varepsilon^2
     \left| \bbE\left(\sqrt{\bar{N}(\varepsilon)}\right) -  \frac{1}{M} \sum_{i=1}^M \sqrt{\bar{N}_i} \right|^2
     +
     2\varepsilon^2
     \bbE(\bar{N}(\varepsilon))
     \leq 
     2({\rm error}^2_{M,\varepsilon,\delta_g} +  {\rm diam}(D)^2)
     .
 \end{equation*}
 Then, the elementary estimate 
that $\sqrt{\sum_{i=1}^{M} c_i} \leq \sum_{i=1}^{M} \sqrt{c_i}$
for any positive numbers $c_i$, $i=1,\ldots,M$, 
implies 
 \begin{equation}\label{eq:estimates_0}
 \sum_{i=1}^{M}
\bar{N}_{i}
\leq 
2 M^2 \varepsilon^{-2}
({\rm error}^2_{M,\varepsilon,\delta_g} +  {\rm diam}(D)^2)
.
\end{equation}
 We choose the parameters $\delta_g = \varepsilon$ and $M=\lceil \varepsilon^{-2}\rceil$. 
 The assertion~\eqref{eq:est_N_i_Laplace_part}
 follows by inserting the expression for ${\rm error}^2_{M,\varepsilon,\delta_g}$
 from~\eqref{eq:NN_ext_dist_Laplace_part_err2} into the previous estimate.
 Define the ReLU DNN $ \phi_{1,\bar\delta}$ by its realization
 \begin{equation*}
 \phi_{1,\bar\delta}
 :=
 \frac{1}{M}\sum_{i=1}^M(\phi_{g,\delta_g}(\bar{X}_{\bar{N}_i}))
 .
 \end{equation*}

 Then, as a consequence of~\eqref{eq:Laplace_exct_dist_est1} there exists a constant
 $C'>0$, which only depends on $\|\Delta g\|_{L^\infty(D)}$, $\|g\|_{L^\infty(D)}$,
 and on ${\rm diam}(D)$ 
 such that
 \begin{equation*}
 \sqrt{
 \int_D
 | \bbE_x(g(X_{\tau_D})) -  \phi_{u,\bar\delta}(x) |^2
 {\rm d}x
 }
 \leq 
 C'\varepsilon 
 (1 + \sqrt{|D|})
 .
 \end{equation*} 
 We choose $\varepsilon = \bar{\delta}/(C'(1+\sqrt{|D|}))$, 
 which proves the assertion of this proposition.
\end{proof}

\begin{proposition}\label{prop:NN_approx_Poisson_part}
    Let $d\geq 3$.
	Suppose that $f$ is Lipschitz continuous on $\overline{D}$,  $x\mapsto {\rm dist}(x,\partial D)$
	can be realized by a ReLU DNN $\phi_{\rm dist}$,
	and for any $\delta_f>0$
	there exists a ReLU DNN $\phi_{f,\delta_f}$ 
	such that
	\begin{equation*}
	\| f - \phi_{f,\delta_f}\|_{L^\infty(D)}
	\leq
	\delta_f
	.
	\end{equation*}
	For every $\bar{\delta}\in (0,1)$, there exists a ReLU DNN $\phi_{2,\bar{\delta}}$
	such that
	\begin{equation*}
	\sqrt{\int_D \left|\bbE_{x}\left(\int_0^{\tau_D}f(X_{t}){\rm d}t\right) - \phi_{2,\bar{\delta}}(x)\right|^2{\rm d}x}
	\leq  
	\bar{\delta}
	.
	\end{equation*}
	Furthermore, there exist $M_1=M_2=\lceil c  \bar{\delta}^{-2} (1 + |D| ^2)\rceil $, $\bar{N}_i$, unit vectors $Y_{i,k}$, and elements of the unit ball $y_{i,j,k}$,  $i=1,\ldots,M_1$,$k=1,\ldots,\bar{N}_i$, $i=1,\ldots,M_2$ such that for every $x \in D$, 
	\begin{equation*}
\phi_{2,\bar{\delta}} (x) 
=
\frac{1}{M_1}\sum_{i=1}^{M_1}
\sum_{k=1}^{\bar{N}_i}
\tilde{\times}_{\tilde\delta}\left(\tilde{\times}_{\tilde\delta}(\phi_{\rm dist}(\bar{X}_{k-1}),\phi_{\rm dist}(\bar{X}_{k-1})),
\frac{1}{M_2}\sum_{j=1}^{M_2}
\phi_{f,\delta_f}(\bar{X}_{k-1} + \phi_{\rm dist}(\bar{X}_{k-1})y_{i,j,k} )
\right), 
\end{equation*}
	where $\bar{X}_k = \bar{X}_{k}(x,Y_{i,1},\ldots,Y_{i,k})$.
        The accuracy $\delta_f$ of the ReLU DNN $\phi_{f,\delta_f}$
        satisfies $\delta_f = c' \bar{\delta} /(1+|D|)$ and
        the accuracy $\tilde{\delta}>0$ of the ReLU DNN $\tilde{\times}_{\tilde{\delta}}$
        satisfies $\tilde{\delta} = c'' \bar{\delta}^5/ (1+ |D|)^5 $.
	The numbers $\bar{N}_i$ satisfy that
	\begin{equation}\label{eq:prop:est_sum_N_i}
	\sum_{i=1}^{M_1}
	\bar{N}_i
	\leq 
	C M_1^2 ( M_1^2  + |D|)
	.
	\end{equation}
	The constants $c,c',c'',C>0$ depend only on $\|f \|_{L^\infty(D)}$ and on ${\rm diam}(D)$.
\end{proposition}

\begin{proof}
Let $\varepsilon \in (0,1)$ and $\delta_f \in (0,1) $ be arbitrary and sufficiently small. The value of these two numbers will be chosen at a later stage in the following proof.
The effect of the approximation of the right hand side $f$ by $\phi_{f,\delta_f}$ 
	is estimated by Lemma~\ref{lem:tau_r_ball_exp}, i.e.,
	\begin{equation}\label{eq:est_RHS_appr_NN}
	\sup_{x\in D}
	\left| \bbE_x\left( \int_0^{\tau_D } f(X_t) {\rm d}t\right)  - \bbE_x\left(\int_0^{\tau_D}\phi_{f,\delta_f}(X_t){\rm d}t\right)\right|
	=
	\bbE_x\left(\tau_D\right)\|f-\phi_{f,\delta_f}\|_{L^\infty(D)}
	\leq  
	\frac{ {\rm diam}(D)^2 /4  }{d} \delta_f, 
	\end{equation}
	where the domain $D$ may be embedded into a ball with radius ${\rm diam}(D)/2$ in order to apply Lemma~\ref{lem:tau_r_ball_exp}.
	
	Define the random number $\bar{N}(\varepsilon) = \sup_{x\in D} N(\varepsilon)$. 
	By Proposition~\ref{prop:error_of_WOS}, 
	\begin{equation*} 
	\sup_{x\in D}\left|\bbE_x\left( \int_{0}^{\tau_D} \phi_{f,\delta_f}(X_t){\rm d}t \right) 
	-  
	\bbE\left( \sum_{k=1}^{\bar{N}(\varepsilon)}   r_k^2 K_1(\phi_{f,\delta_f}(\bar{X}_{k-1} + r_{k} \cdot))   \right)\right| 
	\leq 
	\| \phi_{f,\delta_f} \|_{L^\infty(D)}
	{\rm diam}(D) \varepsilon 
	.
	\end{equation*}
	
	Let us introduce two Monte Carlo estimators $E_{M_1}(\cdot)$, $M_1\in\bbN$,
	and $E_{M_2}(\cdot)$, $M_2\in \bbN$, 
	on the probability space $(\Omega,\calA,\bbP)$.
	The estimators $E_{M_1}(\cdot)$, $M_1\in\bbN$,
	are with respect to
	random variables $\bar{N}(\varepsilon)$
and the sequence $\bar{X}_k$, $k\geq 0$, 
see~\eqref{eq:def_MC_est1}, and satisfy~\eqref{eq:MC_err_est}.
       The estimators $E_{M_2}(\cdot)$, $M_2\in\bbN$, 
       shall approximate the functional $v \mapsto K_1(v)$.
       As a preparation,
       for $d\geq 3$, by~\eqref{eq:boggios_formula}
	\begin{equation*}
K_1(v) 
= 
\frac{\Gamma(1+d/2)}{d \pi^{d/2}}\frac{1}{d-2}
\int_{B(0,1)}
v(y)
(|y|^{2-d} -1){\rm d}y
.
	\end{equation*}
	It holds that
	\begin{equation*}
	\begin{aligned}
	K_1(1)
	=
	\frac{\Gamma(1+d/2)}{d \pi^{d/2}(d-2)}
	\int_{B(0,1)}
	(|y|^{2-d}-1)
	{\rm d}y
	&=
	\frac{\Gamma(1+d/2)}{d \pi^{d/2}(d-2)}
	\left(\frac{1}{2}|\partial B(0,1)|
	- |B(0,1)|\right)
	\\
	&= 
	\frac{\Gamma(1+d/2)}{d \pi^{d/2}(d-2)}
	|B(0,1)|\left( \frac{d}{2}-1\right)
	\\
	&=
	\frac{d/2-1}{d(d-2)}
	=\frac{1}{2d}
	=: \kappa_d
	,
	\end{aligned}
	\end{equation*}
	where we inserted the relation
	$|\partial B(0,1)| = d |B(0,1)|$
	and 
	the value for the volume of the unit $d$-ball, 
	i.e., $|B(0,1)| = \pi^{d/2} /\Gamma(1+d/2)$. 
	Note that $|\partial B(0,1)|$
	denotes the measure of the $d-1$ dimensional 
	unit sphere.
	Thus, 
	\begin{equation*}
	\mu({\rm d}y)
	:=
	\kappa_d^{-1}
	\frac{\Gamma(1+d/2)}{d \pi^{d/2}(d-2)}
	(|y|^{2-d}-1) {\rm d}y
	\end{equation*}
	is a probability measure on $B(0,1)$. 
	The Monte Carlo estimators $E_{M_2}(\cdot)$, $M_2 \in \bbN$, 
	are with respect to the probability measure $\mu$, 
	i.e., for every $M_2\in\bbN$, let $y_i$, $i=1,\ldots,M_2$,
	be independent random variables distributed according to $\mu$
	such that for every
	$v\in L^{2}(B(0,1),\mu)$, 
	\begin{equation}\label{eq:MC_err_est2}
	\sqrt{
	\bbE(|
	K_1(v)
	- \kappa_d
	E_{M_2}(v)
	|^2
	}
	=
	\kappa_d\frac{\left\|v - \int_{B(0,1)} v \mu({\rm d}y)\right\|_{{L^2(B(0,1),\mu}}}{\sqrt{M_2}}
	,
	\end{equation}
	where 
	\begin{equation}\label{eq:MC_est_2}
	E_{M_2}(v) = 
	\frac{1}{M_2}
	\sum_{j=1}^{M_2}
	v(y_j).
	\end{equation}

	We split the error into the contributions from the approximation of the expectation $\bbE_x(\cdot) $ 
	and from the approximation of the integral in the functional $v\mapsto K_1(v)$, i.e., 
	by the triangle inequality
	\begin{equation*}
	\begin{aligned}
	&
	\left\|\bbE
	\left( \sum_{k=1}^{\bar{N}(\varepsilon)}
	r_k^2 
	K_1(\phi_{f,\delta_f}(\bar{X}_{k-1} + r_{k}\cdot))
	\right) 
	- E_{M_1}\left( 
	\sum_{k=1}^{\bar{N}(\varepsilon)}
	r_k^2 
	\kappa_d E_{M_2}(\phi_{f,\delta_f}(\bar{X}_{k-1} + r_{k}\cdot))
	\right)\right\|_{L^2(\Omega, \bbP; L^2(D))}
	\\
	&\leq 
   \underbrace{\left\|\bbE
	\left( \sum_{k=1}^{\bar{N}(\varepsilon)}
	r_k^2 
	K_1(\phi_{f,\delta_f}(\bar{X}_{k-1} + r_{k}\cdot))
	\right) 
	- \bbE\left( 
	\sum_{k=1}^{\bar{N}(\varepsilon)}
	r_k^2 
	\kappa_d E_{M_2}(\phi_{f,\delta_f}(\bar{X}_{k-1} + r_{k}\cdot))
	\right)\right\|_{L^2(\Omega, \bbP; L^2(D))}}_{I:=}
	\\
	&
	+
	\underbrace{\left\|\bbE
	\left( \sum_{k=1}^{\bar{N}(\varepsilon)}
	r_k^2 
	\kappa_d E_{M_2}(\phi_{f,\delta_f}(\bar{X}_{k-1} + r_{k}\cdot))
	\right) 
	- E_{M_1}\left( 
	\sum_{k=1}^{\bar{N}(\varepsilon)}
	r_k^2 
	\kappa_d E_{M_2}(\phi_{f,\delta_f}(\bar{X}_{k-1} + r_{k}\cdot))
	\right)\right\|_{L^2(\Omega, \bbP; L^2(D))}}_{II:=}
	.
	\end{aligned}
	\end{equation*}
	The Monte Carlo estimators $E_{M_1}(\cdot)$ and $E_{M_2}(\cdot)$ are independent
	and also independent from the sequence of random directions $Y_k$, $k\geq 1$, introduced in~\eqref{eq:def_Xbar_k}.

	We estimate by~\eqref{eq:MC_err_est2} that for every $x\in D$, 
	\begin{equation*}
	\begin{aligned}
   &\left\| 
   \sum_{k=1}^{\bar{N}(\varepsilon)}
   r_k^2 
   \left(
   K_1(\phi_{f,\delta_f}(\bar{X}_{k-1} + r_{k}\cdot))
    -
    \kappa_d E_{M_2}(\phi_{f,\delta_f}(\bar{X}_{k-1} + r_{k}\cdot))
   \right)
   \right\|_{L^2(\Omega,\bbP_x )}
	\\
	&\leq
	\left \|
	\sum_{k=1}^{\bar{N}(\varepsilon)}
	r_k^2 
	\kappa_d\frac{\|\phi_{f,\delta_f}(\bar{X}_{k-1} + r_{k}\cdot)\|_{L^2(B(0,1),\mu)}}{\sqrt{M_2}}
	\right\|_{L^2(\Omega,\bbP_x)}
	,
	\end{aligned}
	\end{equation*}
	where we used the indenpendence of $E_{M_2}(\cdot)$ from $Y_k$, $k\geq 1$.
	Furthermore, by~\eqref{eq:mean_sq_cont_stop_time} and Lemma~\ref{lem:tau_r_ball_exp} for every $x\in D$,
	\begin{equation*}
	\bbE\left(
	\sum_{k\geq 1} r_k^2
	\right)
	=
	\bbE_x\left(
	\sum_{k\geq 1} |X_{\calI(k)} - X_{\calI(k-1)}|^2
	\right)
	=
	\bbE_x(\tau_D)
	\leq  \frac{{\rm diam(D)^2/4 - r_0^2}}{d} ,
	\end{equation*}
	where we recall that $r_k$ depends on $x$ via $r_k ={\rm dist}(\bar{X}_{k-1},\partial D)$, $k\geq 1$,
	and $\bar{X}_k$ has the same distribution as $X_{\calI(k)}$, $k\geq 0$.
	Thus, 
	\begin{equation}\label{eq:est_I}
I
\leq 
\sqrt{|D|} 
\kappa_d
 \frac{{\rm diam(D)^2/4}}{d}
 \frac{\|\phi_{f,\delta_f}\|_{L^\infty(D)}}{\sqrt{M_2}}
 .
	\end{equation}

To estimate $II$, note that $|E_{M_2}(\phi_{f,\delta_f}(\bar{X}_{k-1} + r_{k}\cdot))|\leq \|\phi_{f,\delta_f}\|_{L^\infty(D)} $.
Since $r_{k}^2/d = \bbE_{X_{\calI(k-1)}}(\tau_k)$ by Lemma~\ref{lem:tau_r_ball_exp}, Jensen's inequality implies
\begin{equation*}
\begin{aligned}
\bbE\left( 
\left( 
\sum_{k\geq 1}
r_k^2
\right)^2
\right)
&=
d
\bbE_x\left( 
\left( 
\sum_{k\geq 1}
\bbE_{X_{\calI(k-1)}}(\tau_k)
\right)^2
\right)
\\
&\leq
d
\bbE_x 
\left( 
\left( 
\sum_{k\geq 0}
\tau_k
\right)^2
\right)
= 
d
\bbE_x(\tau_D^2)
\leq
2 {\rm diam}(D)^4\left(2 - \frac{d}{d+2}\right)
,
\end{aligned}
\end{equation*}
where we used that $\bbE_0(\tau_{B(0,r)}^2) = 2 r^4(2/d - 1/(d+2))$ for any $r>0$,
which follows for example from~\cite[Equation~(B)]{Getoor_1961}.
We conclude that 
\begin{equation} \label{eq:est_II}
II\leq \kappa_d \sqrt{|D|}\|\phi_{f,\delta_f}\|_{L^\infty(D)} 
{\rm diam}(D)^2
\sqrt{
\frac{4 - {2d}/({d+2})}{M_1}}
.
\end{equation}
Moreover, by Lemma~\ref{lem:exp_of_N} it holds that 
\begin{equation}\label{eq:est_III} 
 \varepsilon \left\|\bbE\left(\sqrt{\bar{N}(\varepsilon)}\right) -  E_{M_1}\left(\sqrt{\bar{N}(\varepsilon)}\right)\right\|_{L^2(\Omega,\bbP)} 
	\leq 
	\frac{{\rm diam}(D)}{\sqrt{M_1}}
	.
\end{equation} 

We combine the estimates~\eqref{eq:est_I}, \eqref{eq:est_II}, and~\eqref{eq:est_III}, which results in 
\begin{equation}\label{eq:NN_ext_dist_Poisson_part_err2}
\begin{aligned}
&\bbE\left( 
\int_D \left|
\bbE_x\left( \int_0^{\tau_D } f(X_t) {\rm d}t\right)
-   
 E_{M_1}\left( 
 \sum_{k=1}^{\bar{N}(\varepsilon)}
 r_k^2 
 E_{M_2}(\phi_{f,\delta_f}(\bar{X}_{k-1} + r_{k}\cdot))
 \right)
 \right|^2 {\rm d}x
\right.
\\
 &\quad+
 \left.
 \vphantom{
 \int_D \left|
\bbE_x\left( \int_0^{\tau_D } f(X_t) {\rm d}t\right)
-   
 E_{M_1}\left( 
 \sum_{k=0}^{\bar{N}(\varepsilon)-1}
 r_k^2 
 E_{M_2}(\phi_{f,\delta_f}(\bar{X}_k + r_{k+1}\cdot))
 \right)
 \right|^2 {\rm d}x
 }
 \varepsilon^2 
 \left|
 \bbE\left(\sqrt{\bar{N}(\varepsilon)}\right) -  E_{M_1}\left(\sqrt{\bar{N}(\varepsilon)}\right)
 \right|^2
\right)
\\
&
\leq 
4|D|\frac{ {\rm diam}(D)^4 /16  }{d^2} \delta_f^2
+
4|D|{\rm diam}(D)^2 \|\phi_{f,\delta_f}\|_{L^\infty(D)}^2  \varepsilon^2
\\
&\quad +
4
|D| \frac{{\rm diam}(D)^4/16  }{d^2}
\frac{\|\phi_{f,\delta_f}\|^2_{L^\infty(D)}}{M_2}
+
8 |D|
\kappa_d
\frac{\|\phi_{f,\delta_f}\|^2_{L^\infty(D)} {\rm diam}(D)^4(2 - d/(d+2))}{ M_1}
+
\frac{{\rm diam}(D)^2}{M_1}
\\
&=: {\rm error}_{M_1,M_2, \varepsilon,\delta_f}^2
.
\end{aligned}
\end{equation}
Recall the elementary observation that for a positive random variable $Z$ and $c>0$ such that $\bbE(Z) \leq c$, 
there exists a measurable set $A\subset \Omega$ satisfying $\bbP(A)>0$ and $Z(\omega)\leq c$ for every $\omega \in A$.
Thus, for every $M_1,M_2\in\bbN$ and for every $\varepsilon \in (0,1)$ there exists 
$\bar{N}_{i}$, $Y_{i,k}$,  $y_{i,j,k}$, $i=1,\ldots,M_1$, $k=1,\ldots,\bar{N}_{i}$, $j=1,\ldots,M_2$,
such that
\begin{equation}\label{eq:total_error}
\int_D\left|
\bbE_x\left(
\int_0^{\tau_D}f(X_t) {\rm d}x
\right)
- 
\frac{1}{M_1}\sum_{i=1}^{M_1}
\sum_{k=1}^{\bar{N}_i}
r_k^2
\frac{1}{M_2}\sum_{j=1}^{M_2}
\phi_{f,\delta_f}(\bar{X}_{k-1} + r_{k}y_{i,j,k} )
\right|^2 {\rm d}x
\leq 
{\rm error}_{M_1,M_2, \varepsilon,\delta_f}^2
\end{equation}
and
\begin{equation}\label{eq:est_sum_Ni}
\sum_{i=1}^{M_1}
\bar{N}_{i}
\leq 
2 M_1^2 \varepsilon^{-2}
\bbE(\bar{N}(\varepsilon)) 
+ 
2M_1^2\varepsilon^{-2}
{\rm error}_{M_1,M_2,\varepsilon,\delta_f}^2, 
\end{equation}
where the latter estimate follows with the Jensen inequality and the elementary estimate 
that $\sqrt{\sum_{i=1}^{M_1} c_i} \leq \sum_{i=1}^{M_1} \sqrt{c_i}$
for any positive numbers $c_i$, $i=1,\ldots,M_1$, see the derivation of~\eqref{eq:estimates_0}.

Let us define the DNN $\phi_{2,\bar{\delta}}$ 
by its realization, i.e.,
for every $x\in D$
\begin{equation}\label{eq:NN_exact_dist_fct}
\phi_{2,\bar{\delta}} (x) 
=
\frac{1}{M_1}\sum_{i=1}^{M_1}
\sum_{k=1}^{\bar{N}_i}
\tilde{\times}_{\tilde\delta}\left(\tilde{\times}_{\tilde\delta}(\phi_{\rm dist}(\bar{X}_{k-1}),\phi_{\rm dist}(\bar{X}_{k-1})),
\frac{1}{M_2}\sum_{j=1}^{M_2}
\phi_{f,\delta_f}(\bar{X}_{k-1} + \phi_{\rm dist}(\bar{X}_{k-1})y_{i,j,k} )
\right), 
\end{equation}
where $\tilde{\times}_{\tilde{\delta}}$
is the ReLU DNN from Lemma~\ref{lem:NN_prod_scalars} that approximates the product of two scalars with accuracy $\tilde{\delta}$ 
and ${\rm size}(\tilde{\times}_{\tilde{\delta}}) = \calO(\lceil\log(\tilde{\delta}^{-1})\rceil)$.  
The parameters $\varepsilon $, $\delta_f$, $M_1$, and $M_2$ are chosen to equilibrate error contributions in~\eqref{eq:total_error}.
The assertion~\eqref{eq:prop:est_sum_N_i}
 follows by inserting the expression for ${\rm error}^2_{M_1,M_2,\varepsilon,\delta_f}$
 from~\eqref{eq:NN_ext_dist_Poisson_part_err2} into the estimate~\eqref{eq:est_sum_Ni}.
Specifically, we choose $\delta_f = \varepsilon $ and $M_1  = M_2 = \lceil \varepsilon^{-2} \rceil $. 
Thus, by~\eqref{eq:total_error} and~\eqref{eq:prop:est_sum_N_i} we conclude that
\begin{equation}
\begin{aligned}
\label{eq:NN_exact_dist_fct_L2_err}
\sqrt{
\int_D \left|
\bbE_x\left(
\int_0^{\tau_D}f(X_t) {\rm d}x
\right)
- \phi_{2,\bar{\delta}}(x)
\right|^2 {\rm d}x}
&\leq 
{\rm error}_{M_1,M_2,\varepsilon,\delta_f} 
+ \frac{1}{M_1}\sum_{i=1}^{M_1}\bar{N}_{i} \tilde{\delta}
\\
&\leq 
C' \delta_f(1+\sqrt{|D|})
+
C'\delta_f^{-4}(1+ |D|) \tilde{\delta},
\end{aligned}
\end{equation}
where $C'>0$ is a generic constant that only depends on 
${\rm diam}(D)$ and $\|f\|_{L^\infty(D)}$.
Consequently, we choose $\tilde{\delta} = \delta_f^5$ 
and finally $\delta_f = \bar{\delta}/(2C' (1+|D|))$, 
which then also yields 
$\tilde{\delta} =[\bar{\delta}/(2C' (1+|D|))]^5 $.
\end{proof}

\begin{proof}[Proof of Theorem~\ref{thm:main_result}]
In this proof, we also include the approximation of the distance function to the boundary by a ReLU DNN $\phi_{{\rm dist},\delta}$, where $\delta >0$ is still to be chosen.
We may restrict ourselves to the case $d\geq 3$. 
For $d=1,2$, the statement follows by~\cite[Theorem~1]{yarotsky_2017}, 
since the solution $u$ is Lipschitz continuous on $\overline{D}$
as observed in the proof of Proposition~\ref{prop:Feynman--Kac_repr}
and may be extended Lipschitz-continuously to a suitable box that is a superset of $D$.
For every $x \in D$, we define the process $\widetilde{X}_k$, $k\geq 0$, by
\begin{equation*}
\widetilde{X}_k = 
\widetilde{X}_{k-1} 
+ 
Y_k \phi_{\rm dist, \delta}(\widetilde{X}_{k-1})
\quad 
\text{and} 
\quad \widetilde{X}_0 =x.
\end{equation*}
The assumed accuracy of the ReLU DNN $\phi_{{\rm dist},\delta}$ implies 
\begin{equation*}
\begin{aligned}
|\bar{X}_k - \widetilde{X}_k|
&\leq  
|\bar{X}_{k-1} - \widetilde{X}_{k-1}|
+
|{\rm dist}(\bar{X}_{k-1},\partial D) - \phi_{\rm dist, \delta}(\widetilde{X}_{k-1})|
\\
&\leq 
|\bar{X}_{k-1} - \widetilde{X}_{k-1}|
\\
&\quad
+ 
|{\rm dist}(\bar{X}_{k-1},\partial D) -  {\rm dist}(\widetilde{X}_{k-1},\partial D)|
+
| {\rm dist}(\widetilde{X}_{k-1},\partial D) - \phi_{\rm dist, \delta}(\widetilde{X}_{k-1})|
\\
&\leq
2 |\bar{X}_{k-1} - \widetilde{X}_{k-1}| + \delta 
,
\end{aligned}
\end{equation*} 
where we used that the distance function is Lipschitz continuous with Lipschitz constant equal to one.
Thus, for every $x\in D$, 
\begin{equation*}
|\bar{X}_k - \widetilde{X}_k| 
\leq 2^k \delta
\quad \text{and} \quad 
|r_k - \phi_{{\rm dist},\delta}(\widetilde{X}_{k-1})|
\leq 
\delta(1 + 2^{k-1}) 
.
\end{equation*}

By Proposition~\ref{prop:NN_approx_Laplace_part}, 
for every $\delta_1>0$ the function $\phi_{1,\delta_1}$ satisfies that
\begin{equation}\label{eq:Laplace_part_exct_dist_function}
\sqrt{\int_D \left|\bbE_{x}\left(g(X_{\tau_D})\right) - \phi_{1,\delta_1}(x)\right|^2{\rm d}x}
	\leq  
	\delta_1
	.
\end{equation}
Also according to Proposition~\ref{prop:NN_approx_Laplace_part}, $\phi_{1,\delta_1}$
depends on weight parameters
$M= \lceil  c\delta_1^{-2}(1+\sqrt{|D|})\rceil $, $\bar{N}_{i,1}$, $i=1,\ldots,M$, and unit vectors $Y_{i,k}$, $k=1,\ldots,\bar{N}_{i,1}$ $i=1,\ldots,M$.
However, $\phi_{1,\delta_1}$ does not constitute a ReLU DNN here, since the assumption on the distance function in Proposition~\ref{prop:NN_approx_Laplace_part}
is weakened in the theorem to be proved here. 
Define the ReLU DNN $\phi^1_{u,\bar{\delta}}$ by its realization
\begin{equation*}
	\phi^1_{u,\bar{\delta}}(x)
	=
	\frac{1}{M}\sum_{i=1}^M(\phi_{g,\delta_g}(\widetilde{X}_{\bar{N}_i}(x,Y_{i,1},\ldots,Y_{i,\bar{N}_{i,1}}))).
\end{equation*}
For any $x,y\in \overline{D}$
\begin{equation}\label{eq:NN_perturbed_Lip_est_g}
 |\phi_{g,\delta_g}(x) - \phi_{g,\delta_g}(y)|
 \leq
 2\delta_g + L_g |x-y|,
\end{equation}
where $L_g$ denotes the Lipschitz constant of $g$.	
By the triangle inequality and the estimates~\eqref{eq:Laplace_part_exct_dist_function}
and~\eqref{eq:NN_perturbed_Lip_est_g}
\begin{equation*}
\sqrt{\int_D \left|\bbE_{x}\left(g(X_{\tau_D})\right) - \phi^1_{u,\bar{\delta}}(x)\right|^2{\rm d}x}
	\leq  
	\delta_1 
	+ 
	\sqrt{|D|}\left(2\delta_g
	+
	\frac{1}{M}
	\sum_{i=1}^M
	2^{\bar{N}_{i,1}} \delta \right)
	.
\end{equation*}
We equilibrate the error contributions by the choices $\delta_g = \delta_1$
and $\delta = 2^{-C M^2(M + |D|)} \delta_1$, 
where $C>0$ is the constant from~\eqref{eq:est_N_i_Laplace_part}.
Thus, 
\begin{equation}\label{eq:calibrated_err_1}
\sqrt{\int_D \left|\bbE_{x}\left(g(X_{\tau_D})\right) - \phi^1_{u,\bar{\delta}}(x)\right|^2{\rm d}x}
	\leq  
	\delta_1
	(1 + 3\sqrt{|D|} )
	.
\end{equation}

Recall the ReLU DNN $\tilde\times_{\tilde\delta}$ from Lemma~\ref{lem:NN_prod_scalars}, which approximates the product of two scalars on $[-m,m]^2$ to accuracy $\tilde\delta$; 
$m$ may be chosen appropriately 
such that it upper bounds $r_k$, $r_k^2 + \tilde{\delta}$, and $\|\phi_{f,\delta_f}\|_{L^\infty(D)}$.
It satisfies for any $a,c \in [-\sqrt{m}+\tilde{\delta},\sqrt{m}-\tilde{\delta}]$ and $b,d\in [-m,m]$
\begin{equation}\label{eq:NN_perturbed_nested_prod_est}
\begin{aligned}
 |\tilde\times_{\tilde\delta}(\tilde\times_{\tilde\delta}(a,a),b) - \tilde\times_{\tilde\delta}(\tilde\times_{\tilde\delta}(c,c),d) |
 &\leq 
 2\tilde\delta 
 +
 m (|\tilde\times_{\tilde\delta}(a,a)-\tilde\times_{\tilde\delta}(c,c)|
 + |b-d|)
 \\
 &\leq 
 2(1+m)\tilde{\delta} +
 2m^2 |a-c| + m|b-d|
 .
 \end{aligned}
\end{equation}
Moreover, 
for any $x,y\in \overline{D}$
\begin{equation}\label{eq:NN_perturbed_Lip_est}
 |\phi_{f,\delta_f}(x) - \phi_{f,\delta_f}(y)|
 \leq
 2\delta_f + L_f |x-y|,
\end{equation}
where $L_f$ denotes the Lipschitz constant of $f$.
By Proposition~\ref{prop:NN_approx_Poisson_part}, 
for every $\delta_2$ the function $\phi_{2,\delta_2}$ satisfies that
\begin{equation}\label{eq:NN_exact_dist_fct_L2_err_2}
\sqrt{\int_D \left|\bbE_{x}\left(\int_0^{\tau_D}f(X_{t}){\rm d}t\right) - \phi_{2,\delta_2}(x)\right|^2{\rm d}x}
	\leq  
	\delta_2
	.
\end{equation}
Also according to Proposition~\ref{prop:NN_approx_Poisson_part}, $\phi_{2,\delta_2}$
depends on weight parameters
$M_1=M_2=\lceil c \delta_2^{-2}(1+|D|^2)\rceil $, $\bar{N}_{i,2}$, unit vectors $Y_{i,k}$, and elements of the unit ball $y_{i,j,k}$,  $i=1,\ldots,M_1$,$k=1,\ldots,\bar{N}_{i,2}$ $j=1,\ldots,M_2$ 
However, $\phi_{2,\delta_2}$ does not constitute a ReLU DNN here, since the assumption on the distance function in Proposition~\ref{prop:NN_approx_Poisson_part}
is weakened in the theorem to be proved here. 
Recall $r_k={\rm dist}(\bar{X}_{k-1},\partial D)$
and $\bar{X}_{k-1} =\bar{X}_{k-1}(x,Y_{i,1},\ldots, Y_{i,k-1})$.
The estimates~\eqref{eq:NN_perturbed_nested_prod_est} and~\eqref{eq:NN_perturbed_Lip_est} imply 
for $k=1,\ldots,\bar{N}_{i,2}$, $i=1,\ldots,M_1$
\begin{equation}\label{eq:perturbed_Lipschitz_cont_prod_NN}
\begin{aligned}
 &\left|
 \tilde{\times}_{\tilde\delta}\left(\tilde{\times}_{\tilde\delta}(r_k,r_k),
\frac{1}{M_2}\sum_{j=1}^{M_2}
\phi_{f,\delta_f}(\bar{X}_{k-1} + r_{k}y_{i,j,k} )
\right)
-
\tilde{\times}_{\tilde\delta}\left(\tilde{\times}_{\tilde\delta}(\tilde{r}_k,\tilde{r}_k),
\frac{1}{M_2}\sum_{j=1}^{M_2}
\phi_{f,\delta_f}(\tilde{X}_{k-1} + \tilde{r}_{k}y_{i,j,k} )
\right)
\right|
\\
&\leq 
2(1+m)\tilde{\delta}
+
2m^2 |r_k - \tilde{r}_k|
+ 
m (2\delta_f + L_f[|\bar{X}_{k-1} - \widetilde{X}_{k-1} | + |r_{k} - \tilde{r}_{k}|])
\\
&\leq 
2(1+m)\tilde{\delta}
+
2m^2 2^k \delta 
+
m (2\delta_f + L_f (1+2^{k}) \delta)
.
\end{aligned}
\end{equation}
Define the ReLU DNN $\phi^2_{u,\bar{\delta}}$
by its realization
\begin{equation*}
x\mapsto 
\phi^2_{u,\bar{\delta}}(x)
=
\frac{1}{M_1}\sum_{i=1}^{M_1}
\sum_{k=1}^{\bar{N}_{i,2} }
\tilde{\times}_{\tilde\delta}\left(\tilde{\times}_{\tilde\delta}(\widetilde{r}_k,\widetilde{r}_k),
\frac{1}{M_2}\sum_{j=1}^{M_2}
\phi_{f,\delta_f}(\widetilde{X}_{k-1} + \widetilde{r}_{k}y_{i,j,k} )
\right).
\end{equation*}
By the triangle inequality, \eqref{eq:NN_exact_dist_fct_L2_err_2}, and~\eqref{eq:perturbed_Lipschitz_cont_prod_NN}
\begin{equation*}
 \begin{aligned}
 \sqrt{\int_D \left|
 \bbE_x\left( 
 \int_0^{\tau_D}
 f(X_s){\rm d}s 
 \right)
 -
\phi^2_{u,\bar{\delta}}(x)
\right|^2 {\rm d}x}
&\leq
\delta_2
+
C' \sqrt{|D|} \frac{1}{M_1}\sum_{i=1}^{M_1} \bar{N}_{i,2} (\tilde{\delta} + \delta_f  + 2^{\bar{N}_{i,2}}\delta )
,
\end{aligned}
\end{equation*}
where the constant $C'>0$ only depends on $m$ and $L_f$.
We equilibrate the error contributions by the adjustments
$\delta_f = \tilde{\delta} = \delta_2 C^{-1} M_1^{-1}  (M_1 + |D|)^{-1}$
and 
$\delta \leq \delta_f  2^{-C M_1^2 (M_1 + |D|)}$, 
where $C>0$ is the generic constant from~\eqref{eq:prop:est_sum_N_i}.
These adjustments have potentially decreased the already chosen values of $\delta_f$ and $\delta$.
Thus, 
\begin{equation}\label{eq:calibrated_err_2}
 \sqrt{\int_D \left|
 \bbE_x\left( 
 \int_0^{\tau_D}
 f(X_s){\rm d}s 
 \right)
 -
\phi^2_{u,\bar{\delta}}(x)
\right|^2 {\rm d}x}
\leq 
\delta_2(1 + C' \sqrt{|D|}  )
.
\end{equation}
We define the ReLU DNN $\phi_{u,\bar{\delta}}$
by 
\begin{equation*}
\phi_{u,\bar{\delta}} 
=
\phi^1_{u,\bar{\delta}} + \phi^2_{u,\bar{\delta}}
\end{equation*}
and chose the remaining two parameter $\delta_1$ and $\delta_2$
such that 
$\delta_1 = \bar{\delta} /(2 + 6 \sqrt{|D|})$
and 
$\delta_2 = \bar{\delta}/(2 + 2C' \sqrt{|D|}  )$. 
Thus, by the estimates~\eqref{eq:calibrated_err_1} and~\eqref{eq:calibrated_err_2}
\begin{equation*}
\| u - \phi_{u, \bar{\delta}}\|_{L^{2}(D)}
\leq 
\bar{\delta}
.
\end{equation*}

It remains to estimate the size of the DNN $\phi^1_{u,\bar{\delta}}$ and $\phi^2_{u,\bar{\delta}}$.
We will apply Lemmas~\ref{lem:composition_NNs}
and~\ref{lem:additions_NNs} 
in order to estimate the size of the ReLU DNN $\phi^2_{u,\bar{\delta}}$, 
which is defined by addition and composition of ReLU DNNs.
Note that for the chosen parameters, 
it holds that
${\rm size}(\phi_{f,\delta_f}) = \calO(d^a \bar{\delta}^{-5b}(1+|D|^{4.5 b}) )$,
${\rm size}(\phi_{{\rm dist},\delta})= \calO(d^a \bar{\delta}^{-6b}(1+|D|^{6b}) )$,
and ${\rm size}(\times_{\tilde\delta}) = \calO(\lceil\log(\bar{\delta}^{-1})\rceil + \lceil\log(1+|D|)\rceil )$.
In conjunction with~\eqref{eq:prop:est_sum_N_i},the size of the ReLU DNN $\phi_{2,\bar{\delta}}$ is bounded by 
\begin{equation}\label{eq:est_size_NN_1}
\begin{aligned}
{\rm size} (\phi_{2,\bar\delta})
&\leq  
C_1
\sum_{i=1}^{M_1}
\sum_{k=1}^{\bar{N}_{i,2}}
{\rm size}( \tilde{\times}_{\tilde{\delta}}  )  + 
M_2[
k \; {\rm size}({\rm \phi_{{\rm dist}, \tilde{\delta}}}) + {\rm size}(\phi_{f,\delta_f})    ]
\\
&\leq 
C_2
\sum_{i=1}^{M_1}
 \bar{N}_{i,2} [ {\rm size}( \tilde{\times}_{\tilde{\delta}}  )  + M_2 {\rm size}(\phi_{f,\delta_f})   ]
+
 M_2 \bar{N}_{i,2}^2
{\rm size}({\rm \phi_{{\rm dist}, \tilde{\delta}}}) 
\\
&\leq 
C_3
[M_1^3(M_1^1 + |D|) {\rm size}(\phi_{f,\delta_f}) 
+ 
M_1^5(M_1^2 + |D|^2) {\rm size}({\rm \phi_{{\rm dist}, \tilde{\delta}}})]
\\
&\leq
C_4
d^a
\bar{\delta}^{-14-6b } (1+|D|^{14+6b}),
\end{aligned}
 \end{equation}
where $C_1,C_2,C_3,C_4$ are generic constants.
The size of the ReLU DNN $\phi^1_{u,\bar{\delta}}$
will be asymptotically dominated by the size of the ReLU DNN
$\phi^2_{u,\bar{\delta}}$.
\end{proof}

\begin{remark}
The assumption in the previous theorem on the availability of a DNN that approximates the distance function to the boundary may be verified for example in the case that 
$D=B(0,1)\subset\bbR^d$. 
The distance function to the boundary of $D$ is given by $D\ni x\mapsto {\rm dist}(x,\partial D) = 1- |x| $, where we recall that $|\cdot|$ denotes the Euclidean norm.
Then, $\phi_{{\rm dist}, \delta} (x) := 1 - \phi_{\sqrt{}, \delta_1}( \sum_{i=1}^d \phi_{{\rm sq}, \delta_2}(x_i)  ) $, 
where $ \phi_{\sqrt{}, \delta_1} $ is the DNN that is defined in Lemma~\ref{lem:sqrt_NN} and $\phi_{{\rm sq}, \delta_2}$ is the DNN in \cite{yarotsky_2017} 
that approximates the square of a scalar.
It holds that $|x|^2 \leq \sum_{i=1}^d \phi_{{\rm sq}, \delta_2}(x_i) \leq  |x|^2 + d \delta_2$ and we suppose that $\delta_2 \leq 1/d$.
It satisfies the error estimate 
for every $x\in D$
\begin{equation*}
\begin{aligned}
|{\rm dist}(x,\partial D) - \phi_{{\rm dist}, \delta} (x)|
&=
\left|\sqrt{|x|^2} -  \phi_{\sqrt{}, \delta_1}\left( \sum_{i=1}^d \phi_{{\rm sq}, \delta_2}(x_i) \right ) \right|
\\
&\leq 
\left|\sqrt{|x|^2} - \sqrt{\sum_{i=1}^d \phi_{{\rm sq}, \delta_2}(x_i)}\right| 
+ 
\left|\sqrt{\sum_{i=1}^d \phi_{{\rm sq}, \delta_2}(x_i)} - \phi_{\sqrt{}, \delta_1}\left( \sum_{i=1}^d \phi_{{\rm sq}, \delta_2}(x_i) \right ) \right|
\\
&\leq
\left|\sqrt{|x|^2 - \sum_{i=1}^d \phi_{{\rm sq}, \delta_2}(x_i)}\right| 
+ \delta_1
\leq 
\sqrt{d\delta_2}
+
\delta_1.
\end{aligned}
\end{equation*}
Thus, we choose $\delta_1 = \delta/2$ and 
$\delta_2 = \delta^2/(2d)^2 $, 
which implies that $\phi_{{\rm dist},\delta}$
has accuracy $\delta$
and ${\rm size}(\phi_{{\rm dist}, \delta}) 
= \calO(d\lceil\log(\delta_{2}^{-1})\rceil + \lceil\log(\delta_1^{-1})\rceil^2 ) = \calO(d[\lceil\log(\delta^{-1}\rceil^2)+\log(d)]) $, 
see Lemma~\ref{lem:sqrt_NN}
and \cite[Proposition~2]{yarotsky_2017}.
Since here $|D| =\calO(1)$, Theorem~\ref{thm:main_result} holds without the curse of dimension.
\end{remark}

\begin{remark}
In the case that $D$ is a hypercube, for example $D=(-1/2,1/2)^d$, 
the distance function to the boundary of $D$ can be represented exactly by a ReLU DNN with size $\calO(d)$. 
This is easily seen, since the 
the distance function to the boundary of $D$ is given 
by $D\ni x\mapsto 1/2 - \max\{|x_1|,\ldots, |x_d| \}$.
Note that the absolute value of a scalar satisfies $|y| = \sigma(y) + \sigma(-y)$, $y\in \bbR$, 
and 
the maximum of two scalars satisfies
$\max\{y,z\} = \sigma(y-z) + z$, $y,z\in \bbR$.
Since here $|D| =1$, Theorem~\ref{thm:main_result} holds without the curse of dimension.
\end{remark}

\begin{remark}
The size of the ReLU DNN $\phi_{u,\bar{\delta}}$
depends algebraically on the reciprocal of the accuracy in Theorem~\ref{thm:main_result}.
The exponent $12+8b$ may be reduced when a tighter bound on $\bbE(\sup_{x\in D} N(\varepsilon))$
would be available, see Lemma~\ref{lem:exp_of_N}. 
In the literature, the bound $\bbE_x (N(\varepsilon)) = \calO( \lceil d\log(\varepsilon^{-1})\rceil  ) $
was indicated, cf.~\cite{Motoo_1959}. 
However, it did not seem to be obvious to apply the proposed techniques to also interchange 
supremum over $x\in D$ and expectation, which is essential in our approach.
\end{remark}

\section{Conclusions}
We have established the existence of numerical approximations of solutions to elliptic PDEs with boundary conditions by DNNs.
It is common to obtain the weights of the DNN by an optimization procedures on sampled training data. 
The \emph{generalization error} that the DNN has on different data points in the domain may also be controlled and is ideally also free from the curse of dimension.
This has been analyzed for certain parabolic PDEs
on $\bbR^d$ in~\cite{BGJ_2018}.
The extension to PDEs with boundary conditions is subject of future work.
Moreover, our results apply to the Poisson equation
with non-homogeneous Dirichlet boundary conditions but more general elliptic PDEs could be treated by similar methods.

\appendix
\section{Neural network approximation of the square root}
In this appendix we provide a constructive DNN approximation to the square root function that converges at a spectral rate. 
We use this result to establish spectral DNN approximability of the distance function of Euclidean balls but the result may be of independent interest. 
\begin{lemma} \label{lem:sqrt_NN}
	For every $\bar{\delta} \in (0,1)$, there exists a ReLU DNN $\phi_{\sqrt{} , \bar{\delta}}$ 
	such that
	\begin{equation*}
	\sup_{x\in [0,2]}|\sqrt{x} - \phi_{\sqrt{} , \bar{\delta}} (x)|
	\leq 
	\bar{\delta} 
	\end{equation*}
	with ${\rm size}(\phi_{\sqrt{} , \bar{\delta}}) = \calO (\lceil\log(\bar{\delta}^{-1})\rceil^{2} ) $.
\end{lemma}

\begin{proof}
	The idea of the proof is that ReLU DNNs are able to approximate the product of two scalars well, see~\cite{yarotsky_2017}.
	For every $x\in [0,2]$
	and every $n\in\bbN$ define the sequences 
	\begin{equation}\label{eq:appendix_def_sn_cn}  
	s_{n+1} = s_n - \frac{s_n c_n}{2}  
	\quad \text{and} \quad 
	c_{n+1} = c_n^2 \frac{c_n -3 }{4}
	\end{equation}
	with $s_0 = x$ and $c_0 = x-1$.
	This scheme seems to be introduced in~\cite{gower_1958}.
	Following~\cite{gower_1958}, it holds that for every $n\in \bbN$, 
	$1+c_{n+1} = (1+c_n)(1-c_n/2)^2$, 
	which implies by induction that for every $n\in\bbN$
	\begin{equation}\label{eq:appendix_eq_1}
	x(1+c_n) = s_n^2
	.
	\end{equation}
	It is easy to see that $|(c_n -3)/4|\leq 1$, and thus (by induction) for every $n\in\bbN$
	\begin{equation}\label{eq:quad_conv_est}
	|c_n|
	\leq 
	|c_{n-1}|^2
	\leq
	|c_0|^{2^{n-1}}|c_0|^{2^{n-1}}
	=
	|c_0|^{2^n}
	,
	\end{equation}
	which implies with~\eqref{eq:appendix_eq_1} 
	that for every $n\in\bbN$ 
	\begin{equation}\label{eq:appendix_err_est}
	|x-s_n^2|
	\leq 
	|c_0|^{2^n}
	.
	\end{equation}
	Thus, for every $x\in [0,2]$, $s_n \to \sqrt{x}$ as $n\to\infty $. 
	However, this convergence is not uniform with respect to $x\in [0,2]$.
	For that reason, we introduce a shift by $\delta^2$ for some $\delta\in(0,1)$.
	Specifically, we set for every $x\in [0,2]$, 
	\begin{equation*}
	 s_0 = x+\delta^2
	 \quad \text{and} \quad 
	 c_0  = s_0 -1.
	\end{equation*}
   Suppose  that $x\in[0,2]$.
   By~\eqref{eq:appendix_err_est}, for every $n\in\bbN$
   \begin{equation*}
   |\sqrt{x+\delta^2} - s_n|
   \leq 
   \frac{|x+\delta^2-s_n^2|}{x+\delta^2+s_n}
   \leq 
   \frac{|c_0|^{2^n}}{2(x+\delta^2)}
   =
   \frac{(1-(x+\delta^2))^{2^n}}{2(x+\delta^2)}
   \leq 
   \frac{(1-\delta^2)^{2^n}}{2\delta^2}
   .
   \end{equation*}
   The condition ${(1-\delta^2)^{2^n}}/{2\delta^2}\leq \delta$
   is satisfied if $2^n \geq [\log(1/2) + 3\log(\delta^{-1})] \delta^{-2}$, 
   where we used the fact that $\log(1/(1-\delta^2))\geq \delta^2/(1-\delta^2) \geq \delta^2$. 
   Since $|\sqrt{x} - \sqrt{x+\delta^2}|\leq \delta$ for any $x\in [0,2]$, 
   \begin{equation}\label{eq:err_bound_s_n}
   \sup_{x\in [0,2]} 
   |\sqrt{x} - s_n |
   \leq 
   2\delta 
   \quad 
   \text{for}\quad 
   n \geq 
   \frac{\log\left[ \log(1/2) +3 \log(\delta^{-1})\right] + 2 \log(\delta^{-1})}{\log(2)}
   .
   \end{equation}
   
    The second step of the proof is to account for errors that occur in multiplications in the scheme~\eqref{eq:appendix_def_sn_cn}, 
    which are approximated by ReLU DNNs.
    Let $\tilde{c}_n$ and $\tilde{s}_n$, $n\in\bbN$, denote realizations of DNNs that are defined by
    \begin{equation*}
    \tilde{c}_{n} 
    =
    \tilde{\times}_{\varepsilon/2}\left(\tilde{\times}_{\varepsilon/2}( \tilde{c}_{n-1}, \tilde{c}_{n-1}  ), \frac{ \tilde{c}_{n-1}-3 }{4}\right)
    \quad \text{and} \quad 
    \tilde{s}_n = \tilde{\times}_\varepsilon\left(\tilde{s}_{n-1},1-\frac{\tilde{c}_{n-1}}{2}\right)
    \end{equation*}
    with $\tilde{c}_0 = c_0$ and $\tilde{s}_0 = s_0$.
    The DNN $\tilde{\times}_{\varepsilon/2}$ denotes the ReLU DNN from Lemma~\ref{lem:NN_prod_scalars} that approximate the product of two scalars on $[-1,1]^2$ with accuracy $\varepsilon/2$. 
    Thus, it holds that 
    \begin{equation*}
    |\tilde{c}_n|
    \leq 
    |\tilde{c}_{n-1}|^2 + \varepsilon 
    \quad \forall n\in \bbN
    \end{equation*}
    We seek an upper bound of $|\tilde{c}_n|$
    that corresponds to~\eqref{eq:quad_conv_est}.
    Let us assume that $\sqrt{\varepsilon}\leq 1/(N-1)$ for some $N\in\bbN$
    and let $\eta \in (0,1)$ satisfy $(1+\eta)(1-\delta^2)\leq 1$ 
    and additionally let $\eta$ and $\varepsilon$ satisfy 
    $(1+\eta^{-1})\sqrt{\varepsilon}\leq 1$.
    We seek to show that 
    \begin{equation*}
    |\tilde{c}_n|  \leq |c_0|^{2^n - 2^{n-1}+1} + n\varepsilon
    \quad n=1,\ldots,N.
    \end{equation*}
    Let $p_n = 2^n - 2^{n-1} +1$, $n\in\bbN$. It holds that $p_n = 2p_{n-1} -1$, $n\geq 2$.
    Indeed by induction with respect to $n=2,\ldots,N$, under these conditions, by Young's inequality, 
    \begin{equation*} 
    \begin{aligned}
    |\tilde{c}_n|   
   & \leq  
    (|c_0|^{p_{n-1}} + (n-1)\varepsilon   )^2 +\varepsilon
    \\
    &\leq  
    (1+\eta)|c_0| |c_0|^{2p_{n-1} -1} + (1+\eta^{-1}) \sqrt{\varepsilon} \varepsilon^{3/2}(n-1)^2  + \varepsilon
   \\
   & \leq 
    |c_0|^{2p_{n-1} -1} + \sqrt{\varepsilon}(n-1)^2 \varepsilon + \varepsilon
    \\
    &\leq  |c_0|^{p_{n}} + n\varepsilon
    .
    \end{aligned}
    \end{equation*}  
    Since $p_n \geq 2^{n-1}$, $n\in\bbN$, 
    \begin{equation*}
      |\tilde{c}_n|   \leq 
      |c_0|^{2^{n-1}} + n\varepsilon
      \quad n=1,\ldots,N
      .
    \end{equation*}

The following fact, which follows by an elementary application of the fundamental theorem of calculus, 
\begin{equation*}
\left | 
y^2 \frac{y-3}{4} 
- 
z^2 \frac{z-3}{4} 
\right |
\leq
\frac{3}{4} \bar{c}(2-\bar{c}) 
|y-z|
\quad \forall y,z \in [-\bar{c},1]
\end{equation*}
for any $\bar{c}\in (0,1)$, 
implies that 
\begin{equation}\label{eq:err_bound_c_n}
|c_n - \tilde{c}_n|
\leq 
\frac{9}{4}
|c_0|^{2^{n-2}}
|c_{n-1} - \tilde{c}_{n-1}|
+\frac{9}{2} n\varepsilon
\quad n=1,\ldots,N.
\end{equation}
Denote $b_n = |c_n - \tilde{c}_n|$, $n=1,\ldots,N$.
We seek to prove by induction that 
\begin{equation*}
b_n
\leq 
\bar{\varepsilon}
\left( 
1+
\sum_{i=1}^{n-1}
\prod_{j=i}^{n-1}
\frac{9}{4}|c_0|^{2^{j-1}}
\right)
\quad n=1,\ldots,N,
\end{equation*}
where $\bar{\varepsilon} = 9 N \varepsilon/2$.
Indeed, by~\eqref{eq:err_bound_c_n} for $n=2,\ldots,N$,
\begin{equation*}
\begin{aligned}
b_n 
\leq 
\frac{9}{4} |c_0|^{2^{n-2}} b_{n-1} + 
\bar{\varepsilon}
&\leq 
\frac{9}{4} |c_0|^{2^{n-2}}
\bar{\varepsilon}
\left( 
1+
\sum_{i=1}^{n-2}
\prod_{j=i}^{n-2}
\frac{9}{4}|c_0|^{2^{j-1}}
\right) 
+\bar{\varepsilon}
\\
&=
\bar{\varepsilon}
\left( 
1+
\sum_{i=1}^{n-1}
\prod_{j=i}^{n-1}
\frac{9}{4}|c_0|^{2^{j-1}}
\right)
.
\end{aligned}
\end{equation*}
Another tool is the following estimate for any $a,b,c >1$, 
\begin{equation*}
\int_{0}^\infty
a^y c^{-b^y} {\rm d}y
=
\int_{1}^\infty
z^{\log(a)/\log(b)} 
e^{-\log(c) z} \frac{{\rm d}z}{\log(b)}
\leq 
\frac{k!}{\log(c)^{k+1} \log(b)},
\end{equation*}
where $k=\lceil  \log(a)/\log(b)  \rceil $ and we used the transformation $z=b^y$.
The estimate of this integral implies 
\begin{equation*}
 \sum_{i=1}^{n-1}
\prod_{j=i}^{n-1}
\frac{9}{4}|c_0|^{2^{j-1}}
\leq 
\sum_{i=0}^{n-2}
\left( \frac{9}{4}\right)^j
|c_0|^{2^{j}}
\leq 
1 + 
\int_{0}^\infty
\left( \frac{9}{4}\right)^t
|c_0|^{2^{t}}
{\rm d}t
\leq 
1 +
\frac{2}{\log(|c_0|^{-1})^{3} \log(2)}
\leq
1+ \frac{4}{\log(2)} \delta^{-6}
,
\end{equation*}
where we used that $\log(1+x)>x/2$ for every $x\in [0,1/2]$ (assuming $\delta \in (0,\sqrt{1/3}]$).
Thus, 
\begin{equation}\label{eq:c_n-tilde_c_n}
|c_n - \tilde{c}_n| \leq \bar{\varepsilon} (2 + 4 \delta^{-6}/\log(2))
\quad n=1,\ldots,N.
\end{equation}
We can now estimate the total error 
\begin{equation*}
 \begin{aligned}
  |s_n - \tilde{s}_n|
  \leq 
  |s_{n-1}(1-c_{n-1}/2) - \tilde{s}_{n-1}(1-\tilde{c}_{n-1}/2)  |
  +\varepsilon/2
  \leq 
  |c_{n-1} - \tilde{c}_{n-1}| + |s_{n-1} - \tilde{s}_{n-1}|
  +\varepsilon/2,
 \end{aligned}
\end{equation*}
where we used that $s_{n-1}/2\leq 1$ and $(1-\tilde{c}_{n-1}/2)\leq 1$.
The previous estimate~\eqref{eq:c_n-tilde_c_n} implies that 
\begin{equation*}
 |s_n - \tilde{s}_n|
 \leq 
 n [\bar{\varepsilon}(2+4 \delta^{-6}/\log(2) ) + \varepsilon/2] 
 \quad n=1,\ldots,N.
\end{equation*}
In conclusion, combining with~\eqref{eq:err_bound_s_n} we have estimated that 
\begin{equation*}
\begin{aligned}
 \sup_{x\in[0,2]}  |\sqrt{x} - \tilde{s}_N |
 &\leq 
 2\delta 
 +
 \varepsilon
 N [(9/2)N(2+4 \delta^{-6}/\log(2) ) + 1/2] 
 \end{aligned}
 \end{equation*}
 for $\geq (\log[\log(1/2) + 3\log(\delta^{-1})] + 2\log(\delta^{-1})) /2 $.

It is left now to choose the parameters $\delta, \varepsilon, N$, and $\eta$ in a suitable way 
to estimate the total size of the DNN $\tilde{s}_N$.
For the given target accuracy $\bar{\delta}$, we choose 
$\delta=\bar{\delta}/4$
and 
$N = \lceil (\log[\log(1/2) + 3\log(\delta^{-1})] + 2\log(\delta^{-1})) /2  \rceil$.
Thus, there exists a generic constant $C>0$ that neither depends on $\delta$ nor on $\varepsilon$
such that
\begin{equation*}
 \sup_{x\in[0,2]}  |\sqrt{x} - \tilde{s}_N |
 \leq 
 \frac{\bar{\delta}}{2}
 +
 C \varepsilon \delta^{-7}.
\end{equation*}
The choice $\varepsilon \leq \bar{\delta}/2 (\delta/4)^{7}/C$ implies that
\begin{equation*}
 \sup_{x\in[0,2]}  |\sqrt{x} - \tilde{s}_N | \leq \bar{\delta}.
\end{equation*}
Let $\phi_{\sqrt{},\bar{\delta}}$ be the DNN that corresponds to $\tilde{s}_N$, i.e., $\phi_{\sqrt{},\bar{\delta}}(x) = \tilde{s}_N$ 
for every $x\in [0,2]$. 
It readily follows (see also Lemmas~\ref{lem:composition_NNs}
and~\ref{lem:additions_NNs}) that ${\rm size}(\phi_{\sqrt{},\bar{\delta}}) = \calO(N \lceil\log(\varepsilon^{-1})\rceil) = \calO(\lceil\log(\bar{\delta}^{-1})\rceil^2)$, 
which completes the proof of the lemma.
\end{proof}

\end{document}